\documentclass[11pt,a4paper]{amsart}

\newcommand{\figdir}{}
\usepackage[margin=28mm]{geometry}
\usepackage{amsmath,amsfonts,amssymb,amsthm}
\usepackage{xcolor,graphicx,enumerate}
\usepackage{clrscode3e}
\usepackage{calc}
\usepackage{url}
\usepackage{hyperref}
\usepackage[normalem]{ulem}
\usepackage[foot]{amsaddr}

\newtheorem{theorem}{Theorem}
\newtheorem{lemma}{Lemma}

\newcommand{\pdpd}[2]{\frac{\partial #1}{\partial #2}}
\renewcommand{\Re}{\operatorname{Re}}
\renewcommand{\Im}{\operatorname{Im}}

\newcommand{\RR}{{\mathbb{R}}}
\newcommand{\CC}{{\mathbb{C}}}

\newcommand{\TT}{{\mathbb{T}}}
\newcommand{\ZZ}{{\mathbb{Z}}}

\newcommand{\FourierF}{\mathcal{F}}

\newtheorem{proposition}{Proposition}
\theoremstyle{definition}
\newtheorem{definition}{Definition}
\theoremstyle{remark}

\newcommand{\brackI}{{({\rm I})}}
\newcommand{\brackII}{{({\rm II})}}

\definecolor{aogreen}{rgb}{0.0, 0.5, 0.0}

\newcommand{\ReviseRed}[1]{#1}
\newcommand{\ReviseBlue}[1]{#1}
\newcommand{\ReviseSetColorBlack}{}
\newcommand{\ReviseFootnote}[1]{}

\DeclareMathOperator{\sgn}{sgn}

\title{A time-domain preconditioner for the Helmholtz equation}
\author[C.C.\ Stolk]{Christiaan C.\ Stolk}
\address{University of Amsterdam, Korteweg-de Vries Institute, POBox
  94248, 1090 GE Amsterdam, The Netherlands}
\email{C.C.Stolk@uva.nl}
\begin{document}

\begin{abstract}
  Time-harmonic solutions to the wave equation can be computed in the
  frequency or in the time domain. In the frequency domain, one solves
  a discretized Helmholtz equation, while in the time domain, the
  periodic solutions to a discretized wave equation are sought, e.g.\
  by simulating for a long time with a
  time-harmonic forcing term.  Disadvantages of the time-domain
  method are that the solutions are affected by temporal
  discretization errors and that the spatial discretization cannot be
  freely chosen, since it is inherited from the time-domain scheme.
  In this work we address these issues. Given an indefinite linear
  system satisfying certain properties, a matrix recurrence relation is
  constructed, such that in the limit the exact discrete
  solution is obtained.  By iterating a large, finite number of times,
  an approximate solution is obtained, similarly as in a time-domain
  method for the Helmholtz equation. To improve the convergence, the
  process is used as a preconditioner for GMRES, and the time-harmonic
  forcing term is multiplied by a smooth window function.  The
  construction is applied to a compact-stencil finite-difference
  discretization of the Helmholtz equation, for which previously no
  time-domain solver was available. Advantages of the resulting solver
  are the relative simplicity, small memory requirement and reasonable
  computation times.
\end{abstract}

\maketitle

\vspace{-2ex}

\section{Introduction}

Time-harmonic solutions to the wave equation can be computed in the
frequency or in the time domain. In the frequency domain, a discrete
version of the Helmholtz equation is solved. In the time domain, the
periodic solutions to a discrete wave equation with time-harmonic
forcing term are sought. Frequency domain methods involve less degrees
of freedom. However, the indefinite linear systems resulting from
discretizing the Helmholtz equation are often difficult to
solve. Time-domain methods can be attractive because they require
relatively little memory and are easy to implement if a time-domain
solver is available. The introduction of
\cite{appelo2019waveholtz} contains a recent overview of time-harmonic
wave equation solvers.

Time-domain methods are based on the correspondence between Helmholtz
and wave equations. In this paper we will use as example the damped
wave equation
\begin{equation} \label{eq:damped_wave_equation}
  \frac{1}{c^2} \pdpd{^2 u}{t^2} + R \pdpd{u}{t} - \Delta u  = f ,
\end{equation}
where $u = u(t,x)$ is the wave field, $f = f(t,x)$ is the forcing
term, $\Delta$ is the Laplacian, $c(x)$ is the spatially dependent
wavespeed, $R$ is a spatially dependent damping coefficient, and
$x$ is in some domain $\Omega$ with Dirichlet and/or Neumann
boundary conditions. Let $u(x,t), f(x,t)$ and $U(x),
F(x)$ be related by
\begin{equation} \label{eq:time_harmonic_u_and_f}
  u(x,t) = e^{i \omega t} U(x) , \qquad \text{ and }
  f(x,t) = e^{i \omega t} F(x) .
\end{equation}
Then $u$ satisfies a wave equation with forcing term $f$
if and only if $U$ satisfies the following Helmholtz equation with
forcing term $F$
\begin{equation} \label{eq:damped_Helmholtz_equation}
  - \Delta U  - \frac{\omega^2}{c^2} U + i \omega R U = F .
\end{equation}
This equation is supplemented with Dirichlet and/or Neumann boundary
conditions that carry over from those for
(\ref{eq:damped_wave_equation}).

We briefly review some time-domain approaches.
The most basic time-domain method is derived from the
limiting-amplitude principle. This principle states that solutions
$u(x,t)$ to (\ref{eq:damped_wave_equation}) with zero initial
conditions and forcing term
\begin{equation} \label{eq:basic_idea_forcing_term}
  f(x,t) = e^{i \omega t} F(x)
\end{equation}
satisfy
\begin{equation} \label{eq:limiting-amplitude_principle}
  u(x,t)  = e^{i \omega t} U(x) + O(1) , \qquad t \to \infty 
\end{equation}
under certain conditions on the problem, see
\cite{EncyclopediaOfMath_LimitingAmplitudePrinciple} and references
therein.  Thus, if $u(x,t)$ is a (numerical) solution to this initial
boundary-value problem, and $T$ is some large time, measured in
periods, then an approximate solution to the Helmholtz equation is
given by
\begin{equation} \label{eq:limiting_amplitude_approx_sol}
  e^{- i 2 \pi T} u(x, 2\pi \omega^{-1} T ) .
\end{equation}
We will call this the limiting-amplitude approximate solution for
time $T$.
A more advanced method is the
exact controllability method \cite{bristeau1998controllability}. In
this method the periodicity of the solutions is enforced using
optimization. The starting value for the optimization procedure is
typically some partially converged limiting-amplitude solution.
Recently more insights in and some improvements to this method were
obtained \cite{grote2019controllability}, and its parallel 
implementation was studied \cite{grote2020parallel}.
In \cite{appelo2019waveholtz} an optimization approach called
WaveHoltz was introduced. This method again uses a form of
optimization but with a different optimization functional.

An important feature of the methods just described is that they approximate
periodic solutions of a given discrete wave equation. This has
two consequences that are in general not desirable. First, the results
will be negatively affected by both spatial and temporal
discretization errors, while solutions to discrete Helmholtz equations
only have spatial discretization errors.  Secondly, the method is
limited to situations in which a time-domain scheme is available, and
not (directly) applicable if one only has a discretization of
(\ref{eq:damped_Helmholtz_equation}) available, or perhaps only a
linear system with similar properties.

In this paper, we will address both of these shortcomings by 
developing a new time-domain solver for discrete
Helmholtz equations.
The method takes as a starting point a linear system 
\begin{equation} \label{eq:discrete_damped_Helmholtz}
  H U = F ,
\end{equation}
where $H$ is a complex $N \times N$ matrix, such that
\begin{equation} \label{eq:Re_H_requirement}
  \text{$\Re H$ is symmetric} ,
\end{equation}
 and
\begin{equation} \label{eq:Im_H_requirement}
  \text{$\Im H$ is symmetric positive semidefinite} .
\end{equation}
In equation (\ref{eq:discrete_damped_Helmholtz})
$F$ is a vector in $\CC^N$ and
$U$ is the unknown, also in $\CC^N$.
\ReviseRed{The applications we have in mind are discretized
  Helmholtz equations and other
time-harmonic wave equations. For a certain finite difference
discretization of the Helmholtz equation the method will be worked out
in detail and numerical results will be given.}
We next describe the steps involved in the construction of  the new
time-domain solver.

First an $N \times N$ system of second order ODE's
\begin{equation} \label{eq:general_semi-discrete_equation}
  \pdpd{u^2}{t^2} + A u + B \pdpd{u}{t}  = f 
\end{equation}
and a frequency parameter $\omega$ are constructed.  We will look for
time-harmonic solutions with frequency $\omega$ of the system
(\ref{eq:general_semi-discrete_equation}). Note that $\omega$ is in
general {\em not} the physical frequency parameter used to derive
(\ref{eq:discrete_damped_Helmholtz}).  It is a computational
parameter, chosen together with the matrices $A$ and $B$, and it depends
on $H$ in a way to be specified.  Time-harmonic functions
$u = e^{i\omega t} U$ and $f = e^{i \omega t} F$, with
$U, F \in \CC^N$, satisfy (\ref{eq:general_semi-discrete_equation}) if
and only if
\begin{equation} \label{eq:time-harmonic_from_semi-discrete}
  ( - \omega^2 I + i \omega B  + A) U = F .
\end{equation}
Therefore we will choose $A,B$ and $\omega$ such that
\begin{equation} \label{eq:choice_of_ABomega_intro}
  H = -\omega^2 I + i \omega B + A .
\end{equation}
The system (\ref{eq:general_semi-discrete_equation}) plays
the role of a semi-discrete wave equation.

Secondly, this system is time-discretized.
This is done in such a way that time-harmonic solutions of the
discrete-time system are exactly those of the continuous-time system.
I.e.\ if $u_n, f_n$ are related to $U,F$ by $u_n = e^{i\omega n \Delta
  t} U$ and
$f_n = e^{i \omega n \Delta t} F$ then $u_n,f_n$ are solutions to the
discrete-time system if and only if $U,F$ satisfy
(\ref{eq:time-harmonic_from_semi-discrete}).
For this purpose we present two modified leapfrog methods.
This is related to ideas from the papers
\cite{holberg1987computational,tam1993dispersion}, see also the
optimized time-stepping method in \cite{berland2006low}.

Having a time-discretization of
(\ref{eq:general_semi-discrete_equation}) at hand, the third step is
to define a map from a right-hand side $F$ in
(\ref{eq:discrete_damped_Helmholtz}) to an approximation for the
solution $U$.  For this we follow the idea of equation
(\ref{eq:limiting_amplitude_approx_sol}) with one modification, which
is the inclusion of a smooth window function in the time-harmonic
forcing term. 


It is known that the convergence of limiting-amplitude approximate
solutions can be slow in some cases, e.g.\ in case of resonant wave
cavities. Therefore we will not use the approximate solution operator
directly, but use it as a preconditioner in an
iterative solution method such as GMRES or BiCGSTAB (Krylov
accelleration). This is the fourth and last step of our
construction. 
The new approximate solution operator will be called a {\em
  time-domain preconditioner}.
The term time-domain preconditioner was used before in
\cite{zschiedrich2006advanced}. In that work a related problem was
solved, but the resulting method was substantially different.
The above distinction between frequency- and time-domain methods
appears no longer satisfactory
for this method: It is based on time-domain methodology, but solves a
frequency-domain discrete wave equation.


Conditions for stability of the new time-discretization and the
convergence for large $T$ of the approximate solutions to the exact
solution of (\ref{eq:discrete_damped_Helmholtz}) will be established
theoretically, including estimates for the convergence if $\Im H = 0$.
Numerical examples confirm the converge of the
approximate solutions and show the usefulness of Krylov accelleration
in cases where a resonant low-velocity zone is present.
Results depend weakly on the large time parameter $T$ used in the
preconditioner, and a parameter for the window function, i.e.\ there
is a large set of suitable parameter choices.

A few situations that do not fit in the classical time-domain setting,
but can be handled with the new solver are as follows. First one can
use it with discretizations that have been designed specifically for
the Helmholtz equation. For example finite-difference discretizations
that minimize dispersion errors such as those from
\cite{BabuskaEtAl1995} for the 2-D case and
\cite{stolk2016dispersion,sutmann2007compact,TurkelEtAl2013} for the
3-D case. Our examples are about this application. One can
also imagine situations where the physical time-domain model is
complicated to simulate, but reduces to a relatively simple Helmholtz
equation in the frequency domain.  A third use case is in the context
of a multigrid method. In multigrid methods the coarsest level system
still has to be solved by another (non-multigrid) method. It is based
on the linear system one started with and on the choice of
multigrid method. In Helmholtz equations the convergence can be very
sensititve to the choice of the coarse level system and it is
recommended to use certain prescribed discretizations
\cite{StolkEtAl2014,stolk2016dispersion}.

However, the usefulness of the method is not restricted to these
cases, and some of the ideas could also be incorporated into other
time-domain methods.

The contents of the remainder of the paper is as follows. In
section~\ref{sec:method}, the construction of the new solver is
described. In section~\ref{sec:examples} examples of this construction
are given in case $H$ results from certain finite-difference
discretizations. Section~\ref{sec:analysis} contains the theoretical
results. After that, section~\ref{sec:numerical_examples} contains the
numerical examples. So far it was assumed that $\Im H$ is diagonal, in
order for the time stepping method to be
explicit. Section~\ref{sec:generalize} describes a variant of the
method that involves an explicit time stepping method even if $\Im H$
is non-diagonal.  We conclude the main text with a discussion
section. Appendices contain some remarks on the compact-stencil
finite-difference discretization described in
section~\ref{sec:examples} and \ReviseRed{the derivation of a Fourier transform used in
section~\ref{sec:analysis}}.

\section{Method}
\label{sec:method}

In this section we describe in detail the construction of a
time-domain preconditioner for a matrix $H$ satisfying
(\ref{eq:Re_H_requirement}) and (\ref{eq:Im_H_requirement}).
We recall from the introduction
that there are three main steps:
(i) the definition of a suitable second order system of ODE's of the
form (\ref{eq:general_semi-discrete_equation}); 
(ii) the definition of a suitable time-integration method for this system of
ODE's;
(iii) the definition of a linear map that produces approximate
solutions based on the limiting-amplitude principle.
Step (ii), the time-integration method, will be discussed first, since
the properties of the time-integration method affect the choice of the
system of ODE's (\ref{eq:general_semi-discrete_equation}).
Then steps (i) and (iii) and the
application of the method as a preconditioner are discussed.

\subsection{Frequency-adapted time discretizations of
  (\ref{eq:general_semi-discrete_equation})}
\label{subsec:time-discretization}

The leapfrog or basic Verlet method is a standard method to integrate
equations of the form (\ref{eq:general_semi-discrete_equation}) in
case that $B = 0$. It is obtained, basically, by discretizing the second order
time derivative using standard second order finite differences.
To allow for nonzero $B$, the damping term has to be discretized as well.
A standard way to do this is with central differences
\cite{brunger1984stochastic, schlick2010molecular,
 sandvik2018numerical}. This yields the equation
\begin{equation} \label{eq:central_differences_damped_leapfrog_verlet}
  \frac{1}{\Delta t^2} \left( u_{n+1} - 2 u_n + u_{n-1}
  \right)
  + \frac{1}{2 \Delta t} B \left( u_{n+1} - u_{n-1} \right) 
  + A u_n = f_n ,
\end{equation}
from which $u_{n+1}$ can be solved.
Note that $u_n$ denotes the discrete approximation to $u(n \Delta t)$,
and that $u(t) \in \CC^N$.

This leads to an explicit method only if $B$ is diagonal. 
\ReviseRed{In section~\ref{sec:generalize} we
will discuss a variant that results in an explicit method in case $B$
is non-diagonal.}

We will formally define the time-integrator resulting from
(\ref{eq:central_differences_damped_leapfrog_verlet}).

\begin{definition} \label{def:Icd}
  Let
  \begin{equation} \label{eq:K_L_central_differences}
  \begin{aligned}
    K = {}& \Delta t^2 A
    \qquad {}& 
    L = {}& \Delta t B
    \qquad {}& 
    g_n = \Delta t^2 f_n .
  \end{aligned}
\end{equation}
  {\em Central differences damped leapfrog} will be defined as the time
  integrator given by
  \begin{equation} \label{eq:cd_integrator}
    u_{n+1} = I_{\rm cd}(u_n,u_{n-1},f_n)
    := \left( I + \tfrac{1}{2}L \right)^{-1}
    \left( (2 - K) u_n - (I - \tfrac{1}{2}L) u_{n-1} + g_n \right) .
  \end{equation}
\end{definition}

The stability of these methods is \ReviseRed{discussed} in
section~\ref{sec:analysis}. According to Theorem~\ref{th:stability_cd},
$I_{\rm cd}$ is stable if
\begin{equation} \label{eq:stability_cond_K_L}
  \text{$K$ and $L$ are positive semidefinite},
\end{equation}
and
\begin{equation} \label{eq:stability_cond_cd_upper}
  \text{$4I - K$ is positive definite} .
\end{equation}
The latter condition leads to a CFL bound, that will be discussed below.

\medskip

As mentioned in the introduction, we look for discretizations such
that the time-harmonic solutions of frequency $\omega$ of the
discrete-time system are exactly those of the continuous time system
(\ref{eq:general_semi-discrete_equation}). Due to discretization
errors this is not the case for $I_{\rm cd}$. In the next
proposition we will show that the time-harmonic solutions to
(\ref{eq:general_semi-discrete_equation}) satisfy a recursion of the
same form as (\ref{eq:central_differences_damped_leapfrog_verlet}), but with
different choices of $A,B$. From these recursions modified schemes can derived
that have the desired property.

\begin{proposition} \label{prop:reproduce_symbol_at_omega}
  Let $u_n$ and $f_n$ be related to $U,F \in \CC^N$ by
  \begin{equation}
    u_n = e^{i \omega n \Delta t} U , \qquad
    f_n = e^{i \omega n \Delta t} F
  \end{equation}
  and let 
  \begin{equation} \label{eq:define_alpha_beta}
    \begin{aligned}
      \alpha = {}& \frac{(\Delta t \, \omega)^2}{2 - 2 \cos (\omega \Delta t)}
      = \frac{(\Delta t \, \omega)^2}{4 \sin( \frac{\omega \Delta
          t}{2})^2}  , 
      & \qquad \text{ and } \qquad
      \beta = {}& \frac{\omega \Delta t}{ \sin( \omega \Delta t)} .
    \end{aligned}
  \end{equation}
  Then $U,F$ satisfy (\ref{eq:time-harmonic_from_semi-discrete}) if
  and only if $u_n$, $f_n$ satisfy
  \begin{equation} \label{eq:mod_central_differences_damped_leapfrog_verlet}
    \frac{1}{\Delta t^2} \left( u_{n+1} - 2 u_n + u_{n-1}
    \right)
    + \frac{1}{2 \Delta t} \tilde{B} \left( u_{n+1} - u_{n-1} \right) 
    + \tilde{A} u_n = \alpha^{-1} f_n ,
  \end{equation}
  where  
  \begin{equation} \label{eq:redefineABf_cd}
    \tilde{A} = \alpha^{-1} A  , \qquad \text{ and } \qquad
    \tilde{B} = \alpha^{-1} \beta B ,
  \end{equation}
\end{proposition}

\begin{proof}
  To prove the first claim, 
  $\tilde{A}$, $\tilde{B}$ and $\tilde{c}$
  will be constructed such that
  \begin{equation} \label{eq:recursion_tilde}
  \frac{1}{\Delta t^2} \left( u_{n+1} - 2 u_n + u_{n-1}
  \right)
  + \frac{1}{2 \Delta t} \tilde{B} \left( u_{n+1} - u_{n-1} \right) 
  + \tilde{A} u_n = \tilde{c} f_n ,
  \end{equation}
  if and only if (\ref{eq:time-harmonic_from_semi-discrete}).
  Inserting $u_n = U e^{i n \omega \Delta t}$ into
  (\ref{eq:recursion_tilde}), results in 
  \begin{equation}
    \left[ \frac{2 \cos(\omega \Delta t) - 2}{\Delta t^2}
      + \frac{i \, \sin( \Delta t \omega)}{\Delta t} B
      + \tilde{A}
    \right] U e^{i n \omega \Delta t} = \tilde{c} e^{i n \omega \Delta t} F .
  \end{equation}
  Using the definitions of $\alpha$ and $\beta$ to rewrite the
  left-hand side, this is equivalent to
  \begin{equation}
    \left[ - \frac{\omega^2}{\alpha} + i \frac{\omega}{\beta} \tilde{B}
      + \tilde{A} \right] U = \tilde{c} F .
  \end{equation}
  Multiplying by $\alpha$ results in 
  \begin{equation}
    \left[ - \omega^2 + i \omega \frac{\alpha}{\beta} \tilde{B}
      + \alpha \tilde{A} \right] U = \alpha \tilde{c} F .
  \end{equation}
  This is equivalent to (\ref{eq:time-harmonic_from_semi-discrete})
  if $\tilde{c} = \alpha^{-1}$ and $\tilde{A}$ and $\tilde{B}$ are defined
  as in (\ref{eq:redefineABf_cd}).
\end{proof}

Based on the proposition we define the following time integrators.
The time-harmonic solutions with frequency $\omega$ of these
integrators correspond exactly to time-harmonic solutions of
(\ref{eq:general_semi-discrete_equation}).

\begin{definition} \label{def:Iacd}
  {\em Frequency adapted central differences damped leapfrog}  
  will be defined as the time integrator given by
  \begin{equation} \label{eq:acd_integrator}
    u_{n+1} = I_{\rm acd}(u_n,u_{n-1},f_n)
    := \left( I + \tfrac{1}{2}L \right)^{-1}
    \left( (2 - K) u_n - (I - \tfrac{1}{2}L) u_{n-1} + g_n \right) ,
  \end{equation}
  where $K$, $L$ and $g_n$ are given by
  \begin{equation} \label{eq:K_L_frequency_adapted_central_differences}
    \begin{aligned}
      K = {}& \frac{\Delta t^2}{\alpha} A ,
      & \qquad
      L = {}& \frac{\beta \Delta t}{\alpha} B ,
      & \qquad
      g_n = {}& \frac{\Delta t^2}{\alpha} f_n .
    \end{aligned}
  \end{equation}
\end{definition}

The time integrators $I_{\rm cd}$ and $I_{\rm acd}$ are of the same
form with different choices for $K$ and $L$. Therefore for
$I_{\rm acd}$ the stability conditions are again
(\ref{eq:stability_cond_K_L}) and (\ref{eq:stability_cond_cd_upper}),
but now with $K,L$ as in
(\ref{eq:K_L_frequency_adapted_central_differences}).

\subsection{Choice of the semi-discrete system and the parameters
  $\omega$ and $\Delta t$}
\label{subsec:choose_semidiscrete}

We now look for a second order system of ODE's of the form
(\ref{eq:general_semi-discrete_equation}), and a parameter $\omega$
such that the time-harmonic solutions to
(\ref{eq:general_semi-discrete_equation}) satisfy $H U = F$. Recall
that $\omega$ is the computational frequency parameter, that is chosen
in the construction of the algorithm, and not the physical frequency
used in the underlying Helmholtz equation. We assume the
time-integrator $I_{\rm acd}$ is used and will discuss the
choice of the parameter $\Delta t$ as well. An explicit and stable scheme is obtained
when 
\begin{equation} \label{eq:ImH_diagonal}
  \text{$\Im H$ is diagonal} .
\end{equation}
Section~\ref{sec:generalize} treats the case that $\Im H$ is non-diagonal. 

Because of (\ref{eq:choice_of_ABomega_intro}), we set
\begin{align} \label{eq:A_B_from_H}
  A = {}& \Re H + \omega^2 I , \qquad \text{and} \qquad
  B = \omega^{-1} \Im H .
\end{align}
where $\omega$ is still to be determined.

We assume the integrator $I_{\rm acd}$ is used, and discuss under
which conditions this leads to stable and explicit scheme. Clearly, the
resulting discrete time system is explicit if $\Im H$ is diagonal, and
generally not otherwise, cf.\ (\ref{eq:A_B_from_H}) and definition~\ref{def:Iacd}.
The stability conditions are given in
(\ref{eq:stability_cond_K_L}) and (\ref{eq:stability_cond_cd_upper}).
Because $\Im H$ is positive semidefinite, $L$ is also positive
semidefinite. To ensure that $K$ is positive semidefinite, we set
\begin{equation} \label{eq:define_omega_ImH_diagonal}
  \omega = \sqrt{ - \lambda_{\rm min}(\Re H)} .
\end{equation}
(or to a lower bound for $\sqrt{ - \lambda_{\rm min}(\Re H)}$ if this
value is not exactly known).
The other stability condition (\ref{eq:stability_cond_cd_upper}) is a form of the
well-known CFL condition. It implies that the eigenvalues of
$K$ must be less than $4$. For central differences
damped leapfrog integration (not frequency adapted) it implies the condition
\begin{equation}
  \Delta t < \frac{2}{\sqrt{\lambda_{\rm max}(A)}} .
\end{equation}
For the frequency adapted variant it implies, by
(\ref{eq:K_L_frequency_adapted_central_differences}) and
(\ref{eq:define_alpha_beta}) the condition 
$\sin( \frac{\omega \Delta t}{2} )^2 
  < \frac{\omega^2}{\lambda_{\rm max}(A)}$,
hence
\begin{equation} \label{eq:define_Deltat_ImH_diagonal}
  \Delta t
  <
  \frac{2}{\omega} \arcsin( \frac{\omega}{\sqrt{ \lambda_{\rm max}(A) } } ) .
\end{equation}

The choices and requirements (\ref{eq:A_B_from_H}),
(\ref{eq:define_omega_ImH_diagonal}) and
(\ref{eq:define_Deltat_ImH_diagonal}) define a suitable choice of
parameters of the systems of ODEs and the $I_{\rm acd}$ integrator in
case $\Im H$ is diagonal.

\subsection{Time-domain approximate solution operators}
\label{subsec:time-domain_approx_sol}

In this section the time-domain approximate solution operator will be
defined. We also show that a complex approximate solution can be computed
by solving a real time-domain wave equation. This appears to be a standard trick
in the field.
The main novelty is that the formula for the time-harmonic forcing
term (\ref{eq:basic_idea_forcing_term}) is modified so that the
forcing is turned on gradually.

\ReviseRed{%
We first specify the window functions that will be used in the forcing term.
\begin{definition} \label{def:admissible_window_function}
  A function $\chi : \RR \to \RR$ is called a $C^k$ admissible window function if
  the following requirements are satisfied.
  \begin{enumerate}[(i)]
  \item
    $\chi$ is even
  \item
    $\chi$ is $C^k$
  \item
    $\chi(0) = 1$, $\chi(x) = 0$ if $x \ge 1$ and $\chi$ is
    non-increasing on $[0,1]$.
  \end{enumerate}
\end{definition}

The time-domain approximate solution operator associated with the
integrator $I_{\rm acd}$ will be denoted by
$S_T^{\rm acd}$ and is defined as follows.

\begin{definition} \label{def:approx_solver}
Let $A, B$, $\omega$ and $\Delta t$ satisfy the requirements of
subsection~\ref{subsec:choose_semidiscrete}. Let $\chi$ be an
admissible $C^\infty$ window function
\ReviseSetColorBlack
and let $T$ be a positive real
constant, such that $n_{\rm steps} := 2\pi \omega^{-1} T / \Delta t$
is an integer.  For $F \in \CC^N$, let
\begin{equation} \label{eq:define_gradual_forcing}
  f_n = f(n \Delta t) , \qquad
  f(t) =  \ReviseRed{ \chi(1 - \frac{t}{2\pi \omega^{-1} T} ) }  e^{i \omega t} F . 
\end{equation}
The {\em time-domain approximate solution operator for $H$} associated with the
integrator \ReviseRed{$I_{\rm acd}$}
is the linear map $\ReviseRed{S_T^{\rm acd}} : \CC^N \to \CC^N$ defined by
\begin{equation}
  \ReviseRed{S_T^{\rm acd} } F = e^{- i 2\pi T} u_{n_{\rm steps} } , 
\end{equation}
where $u_n$, $n=0,1, \ldots, n_{\rm steps} $ is given by
\begin{equation}
  u_{n+1} = \ReviseRed{I_{\rm acd} } (u_n,u_{n-1}, f_n) , \qquad u_0 = 0 .
\end{equation}
\end{definition}}

In section~\ref{sec:analysis} the convergence of $S_T^{\rm acd} F$
to $H^{-1} F$ will be established under the
assumption that requirements for stability of the the time integrators
discussed in subsection~\ref{subsec:choose_semidiscrete} are
satisfied. Section~\ref{sec:numerical_examples} contains numerical
examples for $S_T^{\rm acd}$.

There is still a lot of freedom to choose the window function $\chi$.
In the numerical examples we will introduce therefore an additional
parameter $\rho$, $0 < \rho \le 1$, such that the window function
is 1 on \ReviseRed{$[0, (1-\rho)]$ and positive but strictly less than one on
$(1-\rho,1)$. On the interval $(1-\rho,1)$ a sine square is
used, which leads to the following form of $\chi$ 
\begin{equation}
  \chi(s) =
  \left\{
    \begin{array}{ll}
      1 & \text{if $s \le \rho$}
      \\
      \sin( \frac{\pi}{2 \rho}(1-s) )^2 & \text{if $1-\rho<s<1$}
      \\
      0 & \text{if $s \ge 1$}.
    \end{array}
  \right.
\end{equation}}%
The parameter $\rho$ will be called the window parameter.
\ReviseRed{%
Based on the convergence analysis in subsections~\ref{subsec:convergence} and
\ref{subsec:speed_of_convergence}, $\rho$ should not only be strictly larger than 0
but also strictly smaller than 1.}

We next show that it is sufficient to solve a real time-domain wave
problem to compute the limiting-amplitude approximate solution. The
argument is given in the continuous case, but is applicable
equally well in the discrete case. 

Let $F$ be a complex right-hand side for the Helmholtz equation and
$u(t,x)$ be the solution to (\ref{eq:damped_wave_equation})
with right-hand side $f(t,x) = F(x) e^{i \omega t}$.  Assuming the
limiting-amplitude principle holds, cf.\
(\ref{eq:limiting-amplitude_principle}), 
an approximate solution to the Helmholtz equation is given by
\begin{equation}
  U(x) \approx e^{-i\omega t} u(t,x) , \qquad \text{for some large
    $t$} .
\end{equation}
The field $\Re u(t,x)$ can be determined by solving the real wave
equation with real forcing term, i.e.\ with forcing term
\begin{equation}
  \Re F(x) e^{i \omega t}
  = \cos(\omega t) \Re F(x) -\sin(\omega t) \Im F(x) .
\end{equation}
From (\ref{eq:limiting-amplitude_principle}) it follows that
\begin{equation} 
  \Re u(t,x) = \Re U(x) \cos(\omega t) - \Im U(x) \sin(\omega t) +
  o(1) , \qquad t \to \infty .
\end{equation}
Approximations to $\Re U(x)$ and $\Im U(x)$ can hence be obtained from $\Re
u(t,x)$ by
\begin{equation} \label{eq:real_evaluation}
  \begin{aligned}
    \Re U(x) \approx {}& \Re u(t,x) ,
    \qquad t = T \frac{2 \pi}{\omega}
    \\
    \Im U(x) \approx {}& \Re u(t,x) ,
    \qquad t = \left( T - \frac{1}{4} \right) \frac{2 \pi}{\omega}
  \end{aligned}
\end{equation}
with $T$ a large integer. Therefore real time-domain simulation is
sufficient.

\ReviseBlue{%
Table~\ref{tab:algorithmSacd} contains an algorithm for computing
$S_T^{\rm acd} U$ using time stepping with real fields. The number of time steps
per period is denoted $n_{\rm P}$. It is assumed $n_{\rm P}$ is a
multiple of 4 to facilitate the computation of $\Re U$ and $\Im U$ as in
(\ref{eq:real_evaluation}). If $n_{\rm P}$ is not a multiple of 4 one
could alternatively take $u_j$ for different choices of $j$, with
relative phase not equal to 0 or 180 degrees, and extract $\Re U(x)$ and $\Im
U(x)$ by solving a $2 \times 2$ linear system.}
\begin{table}
\ReviseBlue{%
\begin{codebox}
\Procname{\proc{TimeDomainPreconditioner}$(F, K, L, T, n_{\rm P})$}
\li $u_{-1} = u_{0} = 0$
\li \For $j = 0, \ldots, T n_{\rm P} -1$
\li     \Do
$u_{j+1} = I_{\rm acd}(u_{j}, u_{j-1},
\chi( 1 - \tfrac{j}{T n_{\rm P}} ) \left[
  \cos( \frac{2 \pi j}{n_{\rm P}}) \Re F
  - \sin( \frac{2 \pi j}{n_{\rm P}}) \Im F
\right] )$
        \End
\li $U = u_{T n_{\rm P}} + i u_{ (T -\frac{1}{4}) n_{\rm P}}$
\li \Return U
\end{codebox}
\caption{Algorithm for a time-domain preconditioner. Computed is
  $U = S_T^{\rm acd} F$, with $U, F \in \CC^N$.
  The number of time steps per period is denoted by $n_{\rm P}$.}
\label{tab:algorithmSacd}}
\end{table}

\subsection{Time-domain preconditioned GMRES}
\label{subsec:Krylov_accelleration}
The approximate Helmholtz solver can be used as a preconditioner for
iterative methods like GMRES of BiCGSTAB.
Without preconditioning, the system to be solved is  $H U = F$.
Applying left-preconditioning means that instead the system
\begin{equation}
  P H U = P F 
\end{equation}
is solved, where $P = S_T^{\rm acd}$. The right-preconditioned system is 
\begin{equation}
  H P V = F .
\end{equation}
The vector $V$ is solved from this system and the solution to the
original problem is then given by $U = P V$.  We propose a solution method
where GMRES is applied to the left-preconditioned system.

\section{Examples}
\label{sec:examples}

The examples we consider are finite-difference discretizations
of the damped Helmholtz equation
\begin{equation} \label{eq:damped_Helmholtz_example}
  - \Delta U - k^2 \left( 1 - i \frac{ \tilde{R} }{\pi} \right) U = F
  .
\end{equation}
Here $\tilde{R}$ is the spatially dependent damping in units of
damping per cycle. We start with standard second order differences and
then apply the method to the discretization from
\cite{stolk2016dispersion}.

The method is not limited to finite-difference
discretizations. Finite-element discretizations can also lead to
linear systems $H U = F$ with $H$ satisfying
(\ref{eq:Re_H_requirement}) and (\ref{eq:Im_H_requirement}).

We will see that in our method, the discrete system is a
discretization of a modified PDE, see (\ref{eq:modified_PDE}) below,
and that the CFL bound for this modified PDE is independent of the
velocity. 

\subsection{Second order finite differences}
It is instructive to start with a simple second order
finite-difference discretization. For readability we describe the
two-dimensional case. The degrees of freedom will be denoted by
$U^{(i,j)}$ (in two dimensions), where $i,j$ are in some rectangular
domain $D \subset \ZZ^2$. The matrix $H$ is defined by the equation
\begin{equation}
  ( H U )^{(i,j)} =
  h^{-2}  \bigg( 4 U^{(i,j)} - U^{(i-1,j)} - U^{(i+1,j)} - U^{(i,j-1)} -
  U^{(i,j+1)} \bigg) - (k^{(i,j)})^2  \left( 1 - i \frac{ \tilde{R}^{(i,j)} }{\pi} \right) U^{(i,j)}
\end{equation}
where $h$ denotes the grid spacing, and Dirichlet boundary conditions
are assumed, i.e.\ $U^{(i,j)} = 0$ if $(i,j) \notin D$.

In this case $\Im H$ is diagonal, and $\omega$ and $\Delta t$ are
chosen based on the values 
\begin{equation}
  \begin{aligned}
  \lambda_{\rm min}(\Re H) = {}& - k_{\rm max}^2
\\
\lambda_{\rm max}(\Re H) = {}& - k_{\rm min}^2 + 8 h^{-2} .
\end{aligned}
\end{equation}
where
\begin{equation}
  k_{\rm max} := \max_{(i,j) \in D} k^{(i,j)} , \qquad \text{ and } \qquad
  k_{\rm min} := \min_{(i,j) \in D} k^{(i,j)} .
\end{equation}
We find that $\omega = k_{\rm max}$
while $\Delta t$ is chosen from (\ref{eq:define_Deltat_ImH_diagonal}) using that
$\lambda_{\rm max}(A) = 8 h^{-2} + k_{\rm max}^2 - k_{\rm min}^2$.
The scheme in the computational time domain, can be written as
\begin{equation} \label{eq:discrete_scheme_2ndOrderFd}
  \begin{aligned}
  {}& \frac{\alpha}{\Delta t^2} (u_{n+1}^{(i,j)} - 2 u_n^{(i,j)} + u_{n-1}^{(i,j)} )
  + \frac{\beta}{2 \Delta t} \frac{(k^{(i,j)})^2 \tilde{R}^{(i,j)}}{\pi k_{\rm max}}
  ( u_{n+1}^{(i,j)} - u_{n-1}^{(i,j)} ) 
  \\
  {}& + h^{-2}  \bigg( 4 u_n^{(i,j)} - u_n^{(i-1,j)} - u_n^{(i+1,j)} - u_n^{(i,j-1)} -
  u_n^{(i,j+1)} \bigg) + ( k_{\rm max}^2 -(k^{(i,j)})^2 )  u_n^{(i,j)} =
  f_n^{(i,j)} .
\end{aligned}
\end{equation}
This follows from
(\ref{eq:mod_central_differences_damped_leapfrog_verlet}), by
multiplying this equation with $\alpha$, and using that $A  = \Re H +
\omega^2$, $B = \omega^{-1} \Im H$ and that in this case $\omega = k_{\rm max}$.

The scheme (\ref{eq:discrete_scheme_2ndOrderFd}) differs substantially
from the standard second order FDTD scheme. In fact, it is a
discretization of the PDE
\begin{equation} \label{eq:modified_PDE}
  \pdpd{^2 u}{t^2}  - \Delta u
  + \frac{k^2 \tilde{R}}{\pi k_{\rm max}} \pdpd{u}{t}  
  + \left( k_{\rm max}^2- k^2 \right) u = f .
\end{equation}
The highest order part of this PDE contains a wave
operator $\pdpd{^2}{t^2}  - \Delta$ with constant velocity, even if
the velocity  in the Helmholtz equation we started with is
non-constant (this velocity is proportional to $1/k(x)$).
Thus we have identified another difference with alternative
time-domain methods like those of
\cite{bristeau1998controllability, grote2019controllability,
  appelo2019waveholtz}, that use a discretization of the standard wave
equation. 

On the other hand, this difference can also be largely
removed. If before applying our method
the operator $H$ would be rescaled, multiplying from the
left and the right by a diagonal matrix with entries $c^{(i,j)}$ on
the diagonal, then the resulting computational time domain scheme
would be close to a standard FDTD scheme.

\subsection{An optimized finite-difference method}
\label{subsec:optimized-fd}
Next we consider a discretization with a 27 point cubic
stencil (in 3-D) that was described in
\cite{stolk2016dispersion}. 
There exists different variants of such compact stencil methods, see
among others \cite{BabuskaEtAl1995,sutmann2007compact,TurkelEtAl2013}.
For each gridpoint, the number of neighbors that the gridpoint
interacts with is relatively small (when compared e.g.\ to
higher order finite elements). Even so, the dispersion errors of these
schemes are also relatively small, as shown in \cite{stolk2016dispersion},
so that these methods can be used with relatively coarse meshes%
\footnote{In some applications relatively coarse meshes are used to
  save on computation time. Reduction of dispersion errors is
  important in this case. In general, other discretization
  errors can also be present, e.g.\ related to the presence
  of variable coefficients.}.

The equation discretized in \cite{stolk2016dispersion} was the real
part of (\ref{eq:damped_Helmholtz_example}).
We will discuss here the situation with constant $k$, for variable $k$
we refer to appendix~\ref{app:opt_fd}.
The stencil is the 27 point cube that we will denote with
$\{ -1,0,1\}^3$. 
For constant $k$, due to symmetry there are four different
matrix coefficients. The values of these coefficients are $h^{-2} f_s(
\frac{h k}{2\pi} )$, for $s=0,1,2,3$ respectively, where the dimensionless
function $f_s$ describes the coefficient as a function of the
dimensionless quantity $\frac{h k}{2\pi} $. To be precise,
\begin{equation}
(\Re H)^{(i,j,l; p,q,r)} =
\left\{
  \begin{array}{ll}
    \frac{1}{h^2} f_{|i-p|+|j-q|+|l-r|}( \frac{h k}{2 \pi} )
    & \text{if $(i-p,j-q,l-r) \in \{-1,0,1\}^3$}
    \\
    0 & \text{otherwise}
  \end{array}
\right.
\end{equation}
In \cite{stolk2016dispersion} the $f_s$ are precomputed functions, in
a class of piecewise polynomial functions, obtained by optimization to
minimize phase errors. 
For the 2-D case, in \cite{BabuskaEtAl1995} explicit
expressions for the optimal functions $f_s$, $s=0,1,2$ were derived.
The imaginary part was discretized as above, i.e.\ it was diagonal
with entries
$\frac{(k^{(i,j,l)})^2 \tilde{R}^{(i,j,l)}}{\pi}$. 

Because $\Im H$ is diagonal, the integrator $I_{\rm acd}$ is used.
The discretization in \cite{stolk2016dispersion} has in addition the
property that it was second order and such that
\begin{equation}
  \text{$\Re H + k^2 I$ is symmetric positive definite} .
\end{equation}
Thus $\omega = k$ in case of constant $k$.
The upperbound for $\Re H$ that determines the maximum for $\Delta t$
is given by
\begin{equation} \label{eq:upperbound_HoptimFd}
  \lambda_{\rm max}(\Re H) = f_0( \frac{k h}{2\pi} )
  - 6 f_1( \frac{k h}{2\pi} )
  + 12 f_2( \frac{k h}{2\pi} )
  - 8 f_3( \frac{k h}{2\pi} ) .
\end{equation}
This bound is typically somewhat less than the bound associated with
second order finite differences with the same parameters.
It is straightforward to follow the recipe of
subsection~\ref{subsec:choose_semidiscrete}. This scheme 
can again be considered a discretization of (\ref{eq:modified_PDE}) rather
than of (\ref{eq:damped_wave_equation}).

\section{Analysis}
\label{sec:analysis}

Here theoretical results will be obtained concerning the stability of
the time integrators defined in
subsection~\ref{subsec:time-discretization}, the solutions of these
time-integration schemes and and the convergence of
the approximate solution operators defined in
subsection~\ref{subsec:time-domain_approx_sol}.

The consequences of the stability analysis for the choice of
$\Delta t$ were already discussed in subsection~\ref{subsec:choose_semidiscrete}.

\subsection{Stability}
\label{subsec:stability}
To study the conditions under which (\ref{eq:cd_integrator}) 
and (\ref{eq:acd_integrator})
are stable, we study the growth of the solutions to the homogeneous recursion
\begin{equation} \label{eq:homogeneous_recursion_cd}
  ( 1 + \tfrac{1}{2} L ) u_{n+1} + (- 2+K) u_n + (I - \tfrac{1}{2}  L) u_{n-1} = 0 .
\end{equation}
Throughout we will assume that $K$ and $L$ are symmetric matrices and
that $L$ is positive semidefinite.  The basic idea is to show that the
following energy functions stays bounded
\begin{equation} \label{eq:define_Ecd}
  E_{\rm cd}(n-1/2) = \langle u_n - u_{n-1} , (4I - K)  (u_n - u_{n-1}  ) \rangle
  + \langle u_n + u_{n-1} , K (u_n + u_{n-1}  ) \rangle .
\end{equation}
Here $\langle \cdot , \cdot \rangle$ denotes the standard inner
product. If the energy function is coercive, than the solution also
stays bounded. The following theorem, states the stability conditions
that follow from such an analysis. The conditions under (\ref{it:item_ii_stability_cd})
in the theorem were already mentioned above, see equations
(\ref{eq:stability_cond_K_L}), (\ref{eq:stability_cond_cd_upper}).

This result is not new, see \cite{grote2012explicit}, page 8, for an
earlier discussion of the same energy function.

\begin{theorem} \label{th:stability_cd}
  Let $u_n$, $n \in \ZZ$ be a solution to
  the homogeneous recursion (\ref{eq:homogeneous_recursion_cd}),
  where $K$ and $L$ are real valued symmetric matrices and $L$ is
  positive semidefinite.
  \begin{enumerate}[(i)]
  \item
    If $K$ and $4I - K$ are positive definite then $E_{\rm cd}$ is equivalent
    to a norm on $\RR^{2N}$. If $L = 0$ then $E_{\rm cd}$ is conserved and
    solutions remain bounded if $t \to \pm \infty$. If $L$ is positive
    semidefinite then $\Delta E_{\rm cd}(n) \le 0$
    and solutions remain bounded if $t \to \infty$.
  \item \label{it:item_ii_stability_cd}
    If $K$ is positive semidefinite and $4I - K$ is positive definite 
    and $K$ has one or more zero eigenvalues then $E_{\rm cd}$ is not equivalent to a
    norm. If $L = 0$ then $E_{\rm cd}$ is
    conserved and solutions grow at most linearly if $t \to \pm \infty$.
    If $L$ is positive semidefinite, then $\Delta E_{\rm cd}(n) \le 0$ and
    solutions grow at most linearly if $t \to \infty$.
  \end{enumerate}
\end{theorem}

\begin{proof}
  To prove the result, a bound is derived for
  \begin{equation} 
    \Delta E_{\rm cd}(n) = E_{\rm cd}(n+1/2) - E_{\rm cd}(n-1/2) . 
  \end{equation}
  Using that $K$ is symmetric in combination with basis rules for
  standard products results in 
\begin{equation}
  \begin{aligned}
  \Delta E_{\rm cd}(n) =  {}&
  \langle u_{n+1} - 2 u_n + u_{n-1} , (4I - K) ( u_{n+1} 
  - u_{n-1} ) \rangle 
  \\
  {}& +
  \langle u_{n+1} - u_{n-1} , K ( u_{n+1} + 2 u_n
  + u_{n-1} ) \rangle .
\end{aligned}
\end{equation}  
By using the recursion equation in two places results one obtains the estimate
\begin{equation}
  \begin{aligned}
    \Delta E_{\rm cd}(n) = {}&
  \langle -K u_n - \tfrac{1}{2} L(u_{n+1} - u_{n-1} ) , (4I - K) ( u_{n+1} 
  - u_{n-1} ) \rangle
  \\
  {}& + 
  \langle u_{n+1} - u_{n-1} ,
  4 K u_n + K ( -K u_n - \tfrac{1}{2}  L(u_{n+1} - u_{n-1} ) ) \rangle
  \\
  = {}&
  - 2 \langle u_{n+1} - u_{n-1} ,
  L  ( u_{n+1} - u_{n-1} ) \rangle
  \\
  \le {}& 0 ,  
\end{aligned}
\end{equation}  
where the last inequality holds because by assumption $L$ is positive semidefinite.
The theorem follows from this result.
\end{proof}

\subsection{Eigenvalue analysis}

  There are two cases for the
  eigenvalue analysis of the solutions to
  (\ref{eq:homogeneous_recursion_cd}), depending on whether $K$ and $L$
  have joint eigenvectors. The standard case in which they do is
  worked out because it is used subsection~\ref{subsec:speed_of_convergence}.
  We also analyse the other case.  In the presence of damping
  ($\Im H \neq 0$, hence $L \neq 0$) the intuition is that undamped
  solutions must correspond to vectors in the kernel of $L$, and hence
  to eigenvalues of $K$, in which case it is clear that an undamped
  solution can exist. The purpose of our analysis was to confirm
  that there are no other undamped solutions, because such solutions
  could affect convergence (even though for the convergence proof the result is
  not needed).
    
  So let $v$ be a
  joint eigenvector of $K$ and $L$ with eigenvalues $k$ and $\ell$.
  Solutions to equation (\ref{eq:homogeneous_recursion_cd})
  of the form $u_n = \Lambda^n v$ exist with $\Lambda$ given by
  \begin{equation} \label{eq:Lambda_pm_stability}
    \Lambda_\pm = \frac{1-k/2}{1+\ell/2}
    \pm \frac{1}{1+\ell/2} \sqrt{ -k + k^2/4 + \ell^2/4}
  \end{equation}
  If $k<0$ or $k>4$, these solutions are exponentially growing.
  Thus, for stability, $K$ and $4I - K$ need to be positive
  semidefinite. In addition we require $4I - K$ to be strictly positive
  definite, to avoid a non-decaying solution of the form $(-1)^n v$
  while $k=4$, $\ell>0$.
  
  Next we assume that $K$ and $L$ do, in general, not
  have joint eigenvalues. We study the presence of linearly growing
  solutions and of solutions that neither grow nor decay using the
  first order form of the recursion (\ref{eq:homogeneous_recursion_cd}), that is given by
\begin{equation} \label{eq:system_dim2N_acd}
  y_{n+1} = \Xi y_n + g_n
  , \qquad
  y_n = \begin{bmatrix} u_{n-1} \\ u_n \end{bmatrix}
  , \qquad
  g_n = \begin{bmatrix} 0 \\ f_n \end{bmatrix} ,
\end{equation}
with 
\begin{equation} \label{eq:define_Xi}
  \Xi =
  \begin{bmatrix}
    0 & I \\
    - (I+L)^{-1} (I-L)
    & (I+L)^{-1} (2I-K)
  \end{bmatrix}
\end{equation}

The properties of the solutions are directly related to the
eigenvalues, (generalized) eigenvectors and the Jordan blocks of
$\Xi$. We next list some properties of $\Xi$ that relate to
eigenvalues $\xi$ of $\Xi$ with $| \xi | = 1$. These correspond to
undamped solutions. 
\begin{theorem} \label{th:conserved_and_linearly_growing_sols_Icd}
  Assume (\ref{eq:stability_cond_K_L}),
  (\ref{eq:stability_cond_cd_upper}). The following properties hold
  for the eigenvalues, (generalized) eigenvectors and Jordan blocks of
  $\Xi$:
\begin{enumerate}[(i)]
\item \label{it:i}
  all eigenvectors are of the form $[u,\xi u]$, with $u \in \CC^N$;
\item \label{it:ii}
  there are no eigenvalues $\xi = -1$;
\item \label{it:iii}
  for $|\xi| = 1$, $\xi \neq 1$ there are no Jordan blocks of size $>1$;
\item \label{it:iv}
  for $|\xi| = 1$, $\xi \neq 1$ the $u$ in $[u,\xi u]$ is in $\ker
L$, and is an eigenvector of $K$ and 
$\xi$ follows from (\ref{eq:Lambda_pm_stability});
\item \label{it:v}
  for $\xi = 1$, there are no Jordan blocks of size $>2$,
an eigenvector is of the form
$v = [ u, u]$ with $u \in \ker K$ and if $w$ is a generalized
eigenvector such that $\Xi w = w
+ v$, $v$ an eigenvector, then $v = [u,u]$ with $u$ also in $\ker L$, and
$w = [ w_1 , w_1 + u ]$ with $w_1 \in \ker K$.
\end{enumerate}
\end{theorem} 

\begin{proof}
  Claim (\ref{it:i}) is obvious.

  Regarding (\ref{it:ii}), the equations $\Xi [ u ,\xi u]^T = \xi [ u ,\xi u]^T$
  with $\xi = -1$ gives $(4I - K ) u = 0$ which is impossible because
  $4I - K$ is positive definite by assumption.
  
  Claim (\ref{it:iii}) follows from the energy growth equations.
  Solution to second order recursion would be of the form
  $u_n = \xi^n w + n \xi^{n-1} v$, and the dominant term in the energy in the
  limit $n \to \infty$ would be 
  $\langle n (1-\xi) v , (4I - K) n (1- \xi) v \rangle$. This term would
  be unbounded as $n \to \infty$ which is impossible, hence any Jordan
  block of size $>1$ must be associated with an eigenvalue $\xi = 1$.

  In the situation of claim (\ref{it:iv}), solutions are of the form $u_n = \xi^n
  v$ with $v \in \CC^N$. Energy must be conserved hence
  $\langle u_{n+1} - u_{n-1} , L (u_{n+1} - u_{n-1}) \rangle = 0$ 
  hence $(\xi - \xi^{-1} ) u \in \ker L$, hence $u \in \ker L$. But then 
  it follows that $K u = ( - \xi^{-1} + 2 - \xi) u$, so that $u$ is a
  simultaneous eigenvector of $K$, and $L$.

  The first part of part (\ref{it:v}) follows from the bound on the
  energy. Jordan blocks of size $>2$ would lead to solutions of the form
  \begin{equation}
    u_n = w^{(2)} + n w^{(1)} + n (n-1) v ,
  \end{equation}
  The dominant contribution to the term
  $\langle u_n-u_{n-1} , (4I-K) (u_n-u_{n-1}) \rangle$ would be 
  $\langle  2(n-1) v , (4I-K) 2 (n-1) v \rangle$, which would be
  unbounded, which is not possible.
  The second part of (\ref{it:v}) follows from the form of $\Xi$. 
  If eigenvector is $V = [ v_1, v_2]^T$, $v_j \in \CC^N$, then $v_1 = v_2$ and
  \begin{equation}
    -(1-L) v_1 + (2I-K) v_1 = (1+L) v_1
  \end{equation}
  hence $K v_1 = 0$.  The third part of (\ref{it:v}) follows from the
  form of $\Xi$.  Working out the first line of equation $\Xi w = w + v$
  yields $w_2 = w_1 + u$, working out the second line gives $L u = 0$.
\end{proof}

\subsection{Convergence of the approximate solution operators $S_T^{\rm
  acd}$}
\label{subsec:convergence}

In this subsection we will show that the time-domain approximate
solution operator $S_T^{\rm acd}$ converges to the
true solution operator.

We first give our conventions regarding the Fourier transform, and
discuss the Fourier transform of functions defined on  $\Delta t \ZZ$.
Let $g(n)$ be a function $\ZZ$. Considering $n$ as a position
variable, we have the following Fourier transform / inverse Fourier
transform pair, denoting the frequency by $\nu$
\begin{equation} \label{eq:FT_Z}
  \begin{aligned}
    \widehat{g}(\nu)
  = {}& \sum_{n \in \ZZ} g(n) e^{-i n \nu} 
 ,  \qquad\qquad
  g(n) = {}& \frac{1}{2\pi} \int_{-\pi}^\pi \widehat{g}(\nu)
    e^{i n \nu} \, d \nu .
\end{aligned}
\end{equation}
Here $\nu$ in the periodic interval $[-\pi,\pi]$ which we will also
denote by
$\TT_{2 \pi}$. We will also denote the Fourier transform of a function
$g$ by $\mathcal{F} g$.

When considering discrete approximations of continuous quantities, it
can be convenient to write $u( n \Delta t)$ instead of $u_n$ and
$f( n \Delta t)$ instead of $f_n$. In other words, $u$ and $f$ are
considered as functions $u(t)$, $f(t)$ with $t  \in \Delta t \ZZ$, and
not as two-sided sequences. The forward and inverse Fourier
transforms between functions on $\Delta t \ZZ$ and functions on
$\TT_{2\pi/\Delta t}$ are defined by
\begin{equation}
  \begin{aligned}
  \widehat{f}(\tau) = \Delta t \, \sum_{t \in \Delta t \ZZ} f(t) e^{-i
    t \tau} , 
  \qquad \qquad
  f(t) = \frac{1}{2\pi} \int_{-\pi/\Delta t}^{\pi /\Delta t} \widehat{f}(\tau)
    e^{i t \tau} \, d \tau .
\end{aligned}
\end{equation}

For functions of $t \in \Delta t \ZZ$ the Dirac delta function will be
defined including a factor $\Delta t^{-1}$, i.e.\
\begin{equation} \label{eq:delta_fun_DeltatZ}
  \delta(t) = \left\{
    \begin{array}{ll}
      \frac{1}{\Delta t} & \text{if $t=0$}\\
      0 & \text{otherwise}
    \end{array}
  \right.
  , \qquad\qquad
  \FourierF \delta (\tau) = 1 .
\end{equation}
For function $g(t), h(t)$, $t \in \Delta \ZZ$, we have
\begin{equation} \label{eq:FourierTransformFG}
  \widehat{g h} =
  \frac{1}{2\pi} \widehat{g} \ast \widehat{h} 
\end{equation}
where $\ast$ denotes convolution on $\TT_{2\pi/\Delta t}$.

In this subsection we will treat $u$ and $f$ as functions on $\Delta
t \ZZ$. 
Associated with the integrator $I_{\rm acd}$ is then a second order
difference operator $C_{\rm acd}$, defined as
\begin{equation} \label{eq:operator_Cacd}
  C_{\rm acd}(u)(t) = \frac{\alpha}{\Delta t^2}
\left[ ( 1 + \tfrac{1}{2} L ) u(t+\Delta t) + (- 2+K) u(t) + (I -
  \tfrac{1}{2} L) u(t-\Delta t) \right] 
\end{equation}
where $K, L$ are as defined in (\ref{eq:K_L_frequency_adapted_central_differences}).
Solutions to the time-integration satisfy
\begin{equation} \label{eq:difference_equation_acd}
  C_{\rm acd}(u)(t) = f(t) .
\end{equation}
The left-hand side can be see as the convolution of $u$ with a kernel
$\Gamma_{\rm acd}$ that is a matrix valued function on $\Delta t \ZZ$.
(The explicit expression for
$\Gamma_{\rm acd}$ is easily derived from (\ref{eq:operator_Cacd}).)
In the Fourier domain the convolution becomes a multiplication and we have
\begin{equation}
  \widehat{\Gamma}_{\rm acd}(\tau) \widehat{u}(\tau) =
  \widehat{f}(\tau) .
\end{equation}

From 
proposition~\ref{prop:reproduce_symbol_at_omega} it follows that
\begin{equation} \label{eq:symbol_equality_property_acd}
  \widehat{\Gamma}_{\rm acd}(\omega) = - \omega^2 I + i \omega B + A .
\end{equation}
In addition we know that $\widehat{\Gamma}$ is $C^\infty$, because
$\Gamma$ is nonzero only for $t \in \{ -\Delta t, 0 , \Delta t \}$.

By $\Phi_{\rm acd}(t)$, $t \in \Delta t \ZZ$ we will denote the causal
Green's function associated with
(\ref{eq:difference_equation_acd}). It is a function
$\Delta t \ZZ \to \CC^{N \times N}$ 
defined as the solution to
\begin{equation} 
  C_{\rm acd} \Phi_{\rm acd} (t) = I \delta(t) ,
  \qquad \text{and} \qquad 
  \Phi_{\rm acd}(t) = 0 \text{ if } t \le  0 ,
\end{equation}
where $\delta(t)$ is as defined in (\ref{eq:delta_fun_DeltatZ}).
Assuming that the conditions in
(\ref{eq:stability_cond_K_L}), (\ref{eq:stability_cond_cd_upper})
hold, $\Phi_{\rm acd}$ grows at most polynomially if $t \to \infty$
(by Theorem~\ref{th:stability_cd}), and 
its Fourier transform $\widehat{\Phi}_{\rm acd}$ is well defined as a matrix valued
distribution on the circle $\TT_{2\pi/\Delta t}$.

We will first show that the approximate solution operator
$S_T^{\rm acd}$ can be expressed as a weighted mean of
$\widehat{\Phi}(\tau)$ around $\tau = \omega$.
\ReviseRed{For this purpose we first define a helper function 
$\psi_{1/T}$, depending on $\chi$ and various parameters.
Let $\theta_T(t)$, $t \in \Delta t \ZZ$ be given by
\begin{equation}
  \theta_T(t) = \chi( \frac{t}{2 \pi \omega^{-1} T} ) ,
\end{equation}
then $\psi_{1/T}$ is defined by
\begin{equation} \label{eq:define_psi_oneoverT}
  \psi_{1/T}(\tau) = \frac{1}{2 \pi} \mathcal{F} \theta_T(\tau) ,
  \qquad \tau \in \TT_{2\pi/\Delta t} .
\end{equation}}%
We have the following result.

\begin{theorem} \label{th:convolution_formula_S_T_acd}
Assume $\psi_\epsilon$ and $\Phi_{\rm acd}$ are as just defined, then 
\begin{equation} \label{eq:convolved_Fourier_transform_acd}
  S_{T}^{\rm acd} =  \psi_{1/T} \ast \widehat{\Phi}_{\rm acd}(\omega) .
\end{equation}
\end{theorem}

In words, $s_T^{\rm acd}$ equals the convolution product of
$\psi_{1/T}$ and $\widehat{\Phi}_{\rm acd}$, evaluated at $\tau =
\omega$. The convolution product is taken on the torus
$\TT_{2\pi/\Delta t}$. 

\begin{proof}
Let $\tilde{T} = 2\pi \omega^{-1} T$.
By linearity, the time-domain approximate solution $S_T(F)$ satisfies
\begin{equation}
  S_T^{\rm acd}(F) = e^{- i \omega \tilde{T}}
  \Delta t \sum_{t \in \Delta t \ZZ, 0 \le t < \tilde{T}}  \Phi(\tilde{T}-t)
  \chi(1 - t/\tilde{T}) e^{i \omega t} F 
\end{equation}
In this formula we easily recognize a Fourier transform
\begin{equation}
  S_T^{\rm acd}(F) = \left[
  \Delta t \sum_{t \in \Delta t \ZZ}  \Phi(\tilde{T}-t) \theta_T(\tilde{T}-t)
  e^{-i \omega (\tilde{T}-t)}  \right]  F
= \left[
  \mathcal{F} ( \theta_T \Phi )(\omega)
  \right] F .
\end{equation}
The equality (\ref{eq:convolved_Fourier_transform_acd}) follows because
the Fourier transform of the product is $\frac{1}{2\pi}$ times
the convolution product of the Fourier transforms,
see (\ref{eq:FourierTransformFG}).
\end{proof}

We continue by studying the family of functions $\psi_{1/T}(\tau)$.
\begin{lemma} \label{lem:psi_hatchi}
  Let $\tilde{T} = 2\pi \omega^{-1} T$.
  Suppose $\chi$ is a $C^\infty$ admissible window function and $\psi_{1/T}$ is defined
  as in (\ref{eq:define_psi_oneoverT}).
The function $\psi_{1/T}(\tau)$ satisfies
\begin{equation} \label{eq:psi_from_hatchi}
  \psi_{1/T}(\tau) =
  \frac{\tilde{T}}{2\pi} \sum_{j \in \ZZ} \widehat{\chi}( (\tau + j
  \frac{2 \pi}{\Delta t}) \tilde{T}) . 
\end{equation}
\end{lemma}
\begin{proof}
  The Fourier transform of the function $s \mapsto \chi(s / \tilde{T})$, $s \in \RR$ is given
  by $\sigma \mapsto \tilde{T} \widehat{\chi}(\sigma \tilde{T})$,
  $\sigma \in \RR$. When $s$ is restricted to $\Delta t \ZZ$
  then the Fourier transform is a function on the torus $\TT_{2\pi /
    \Delta t}$ given by
  \begin{equation}
    \tilde{T} \sum_{j \in \ZZ} \widehat{\chi}( (\tau + j 2 \pi /
    \Delta t) \tilde{T}) ,
    \qquad
    \tau \in \TT_{2\pi / \Delta t} .
  \end{equation}
  This shows (\ref{eq:psi_from_hatchi}).
\end{proof}

Because $\chi$ is $C^\infty$ and compactly supported, its Fourier
  transform $\widehat{\chi}(\sigma)$, defined for
  $\sigma \in \RR$, satisfies
  \begin{equation} \label{eq:decay_property_chi}
    \frac{d^m\widehat{\chi}}{d\sigma^m} = O( |\sigma|^{-N}) , \qquad \sigma \to \pm \infty
  \end{equation}
  for any $m \ge 0$ and $N \ge 0$.
  Furthermore, $\frac{1}{2\pi} \widehat{\chi}(\sigma)$ has integral
  $1$. As a consequence $\psi_{1/\epsilon}$ is an
  approximation to the identity (this is actually proven in the upcoming proof).
  This fact and Theorem~\ref{th:convolution_formula_S_T_acd}
imply the convergence of the time-domain approximate solutions to the
correct solutions.
\begin{theorem} \label{th:convergence_acd}
  If
  \begin{equation} \label{eq:H_non-singular}
    \text{$H = -\omega^2 I + i \omega B + A$ is non-singular}
  \end{equation}
  and $A,B,\omega$ and $\Delta t$ are such that $K$, $L$ defined in
  (\ref{eq:K_L_frequency_adapted_central_differences}) satisfy
  the stability conditions (\ref{eq:stability_cond_K_L}) and (\ref{eq:stability_cond_cd_upper})
  then
  \begin{equation}
    \lim_{T \to \infty} S_T^{\rm acd} = H^{-1} .
  \end{equation}
\end{theorem}

\begin{proof}
  The Fourier transform $\widehat{\Phi}_{\rm acd}$ satisfies
  \begin{equation}
    \widehat{\Gamma}_{\rm acd}(\tau) \widehat{\Phi}_{\rm acd}(\tau) = I .
  \end{equation}
  Note that
  $\widehat{\Gamma}_{\rm acd}(\tau)$ is a $C^\infty$ function and 
  $\widehat{\Phi}_{\rm acd}$ is a distribution and that this equality
  is in the sense of distributions.
  By the assumption (\ref{eq:H_non-singular}) and the property
  (\ref{eq:symbol_equality_property_acd}), the matrix valued function
  $\widehat{\Gamma}_{\rm acd}(\tau)$ is non-singular on a neighborhood
  of $\tau = \omega$. It follows that
  $\widehat{\Phi}_{\rm acd}(\tau)$ is continuous on a neighborhood
  of $\tau = \omega$ and that
  \begin{equation}
    \widehat{\Phi}_{\rm acd}(\omega)
    = \widehat{\Gamma}_{\rm acd}(\omega)^{-1}
    = H^{-1} .
  \end{equation}

  We define a periodic cutoff $\theta_{\rm per}$. This will be
  supported in $[(-\pi-\epsilon)/\Delta t, (\pi+\epsilon)/\Delta t]$
  for some $\epsilon < \pi$ and such that
  \begin{equation}
    \sum_{j=-\infty}^\infty \theta_{\rm per}(\tau + j 2 \pi/\Delta t) = 1
  \end{equation}
  for all $\tau$. 
  From (\ref{eq:convolved_Fourier_transform_acd}) we write $S_T^{\rm
    acd}$ as
  \begin{equation} \label{eq:integral_for_Sacd}
    S_T^{\rm acd} = \int_\RR \theta_{\rm per}(\tau)
    \psi_{1/T}(\tau) \widehat{\Phi}_{\rm acd}(\omega - \tau) \, d \tau .
  \end{equation}

  Next we insert the sum, and distinguish between the cases $j=0$ and
  $j \neq 0$. The case $j=0$ is further split into the set near $\tau
  = 0$ and away from $\tau = 0$ using a cutoff function $\theta_1$
  that is supported on a neighborhood of $\tau=0$ where
  $\widehat{\Phi}(\omega-\tau)$ is smooth and is equal to one on a
  smaller neighborhood of $0$. 
  We thus have three terms
  \begin{equation}
    \begin{aligned}
      S_T^{\rm acd}
      = {}&
      \frac{\tilde{T}}{2\pi} \bigg[ \sum_{j \in \ZZ, j \neq 0}
      \int_\RR \theta_{\rm per}(\tau)
      \widehat{\chi}(\tilde{T}(\tau + j 2 \pi / \Delta t) ) \widehat{\Phi}_{\rm acd}(\omega-\tau) \, d \tau     
      \\
      {}& + 
      \int_\RR \theta_{\rm per}(\tau) (1-\theta_1(\tau))
      \widehat{\chi}(\tilde{T}\tau) \widehat{\Phi}_{\rm acd}(\omega-\tau) \, d \tau
      \\
      {}& +       
      \int_\RR \theta_{\rm per}(\tau) \theta_1(\tau)
      \widehat{\chi}(\tilde{T}\tau) \widehat{\Phi}_{\rm acd}(\omega-\tau) \, d \tau \bigg]       
    \end{aligned}
  \end{equation}
  In the first term on the right hand side, the norms
  $\| \theta_{\rm per}(\tau)
  \widehat{\chi}(\tilde{T}(\tau + j 2 \pi / \Delta t) )\|_{C^k}$
  decay as $\tilde{T}^{-N}$ and $j^{-N}$ for any $N$, because of
  (\ref{eq:decay_property_chi}).
  Since $\widehat{\Phi}_{\rm acd}$ is a distribution of finite order,
  this term goes to zero  as $\tilde{T} \to \infty$.
  The second term goes to zero similarly.

  For the third term, we may assume that $\theta_{\rm per} = 1$ on the
  support of $\theta_1(\tau)$. We may expand
  \begin{equation} \label{eq:expand_phi_around_omega}
  \widehat{\Phi}_{\rm acd}(\omega-\tau)
  = \widehat{\Phi}_{\rm acd}(\omega) + \tau g_1(\tau)
\end{equation}
where $g_1(\tau)$ is bounded by a constant $M$ (actually it is bounded
by the first derivative of
$\widehat{\Phi}_{\rm acd}(\omega-\tau)$ on the support of $\theta_1$).
We write the third term and then apply a
change of variables as follows
\begin{equation}
  \begin{aligned}
    {}& \frac{\tilde{T}}{2\pi}  \int_\RR \theta_1(\tau)
  \widehat{\chi}(\tilde{T} \tau) \bigg[ \widehat{\Phi}_{\rm
    acd}(\omega) - \tau g_1(\tau) \bigg] \, d \tau
  \\
  {}& = \frac{1}{2\pi} \widehat{\Phi}_{\rm acd}(\omega)
  \int_\RR \theta_1(\sigma/\tilde{T}) \widehat{\chi}(\sigma) \, d \sigma
  + \frac{1}{2\pi \tilde{T}} \int_\RR \theta_1(\sigma/\tilde{T})
  \widehat{\chi}(\sigma) \sigma g_1( \sigma/\tilde{T}) \, d \sigma .
\end{aligned}
\end{equation}
The first term on the right hand side converges to
$\widehat{\Phi}_{\rm acd}(\omega)$, since
\begin{equation}
  \frac{1}{2\pi} \int_\RR \theta_1(\sigma/\tilde{T}) \widehat{\chi}(\sigma) \, d
  \sigma
  \quad \to \quad 
  \frac{1}{2\pi} \int_\RR \widehat{\chi}(\sigma) \, d \sigma = 1
  \quad \text{if $\tilde{T} \to \infty$} 
\end{equation}
(the latter equality by the inverse Fourier transform), while the
second term goes to zero, since it can be bounded by 
\begin{equation}
\frac{1}{2\pi \tilde{T}} \| \int_\RR \theta_1(\sigma/\tilde{T})
\sigma g_1( \sigma/\tilde{T}) \widehat{\chi}(\sigma)  \, d \sigma \|
\le
\frac{M}{2\pi \tilde{T}} \int_\RR \left| \sigma \widehat{\chi}(\sigma) \right| \, d \sigma 
\end{equation}
where the latter integral is finite by (\ref{eq:decay_property_chi}). This concludes the
proof.
\end{proof}

\subsection{Speed of convergence}
\label{subsec:speed_of_convergence}

The speed of convergence is the topic of this and the next
subsection. We primarily use a spectral decomposition and
Theorem~\ref{th:convolution_formula_S_T_acd} and restrict to the case
that $\Im H = 0$. In this case $\Re H$ has an orthogonal eigenvalue
decomposition, so that an essentially scalar analysis can be used.

After applying a change of basis, $H$ may be assumed diagonal, given
by
\begin{equation} \label{eq:H_Re_diagonal}
  H = \begin{bmatrix}
    \lambda_1 &&\\
    & \ddots &\\
    && \lambda_N    
  \end{bmatrix}
\end{equation}
where $\lambda_1 \le \lambda_2 \le \ldots \le \lambda_N$.
Matrices
such as $K$, $\Phi_{\rm acd}(t)$ and $\widehat{\Phi}_{\rm acd}(\tau)$
are also diagonal. The diagonal components can be obtained by applying
our approach to the scalar problem in which $H = \lambda$, where we
must take into account that $\omega$ and $\Delta t$ are chosen using
using the set of eigenvalues $\{ \lambda_1, \ldots, \lambda_N \}$ of $H$. So,
for example there is the requirement
\begin{equation}
  \omega^2 \ge | \lambda_1 | 
\end{equation}
for $\omega$. In the analysis of the scalar problem we will obtain estimates uniformly in
$\lambda$. With these considerations it suffices to analyse the scalar problem. 

For the quantities of this scalar problem we use the following notations. 
We write $k(\lambda)$ for the matrix $K$ with the above scalar
problem, and $\phi_{\rm acd}(t,\lambda)$ for the $\Phi_{\rm
  acd}(t)$. For example, according to definition equations
(\ref{eq:A_B_from_H}) and (\ref{eq:K_L_frequency_adapted_central_differences})
the function $k(\lambda)$ is given simply by
\begin{equation}
  k(\lambda) = \frac{\Delta t^2}{\alpha} (\lambda + \omega^2) .
\end{equation}
We will study the decay of the error 
\begin{equation} \label{eq:define_err}
  \text{error}
  = s_{T}^{\rm acd}(\lambda) - \lambda^{-1} ,
\end{equation}
as a function of $\lambda$ and $T$, in the limit $T \to \infty$. We wil
take into account the possibility that $|\lambda|$ is small, obtaining
appropriate uniform estimates.

The first step is an explicit construction of
$\phi_{\rm acd}(t,\lambda)$, which amounts to solving the recursion
given by (\ref{eq:define_Xi}), (\ref{eq:system_dim2N_acd}) for the
case that $L=0$ and $K$ is scalar.

\begin{proposition}
Let and $U(n)$ be the unit step function of
(\ref{eq:define_unit_step}). In case $0<k<4$
\begin{equation} \label{eq:scalar_phi_acd}
  \phi_{\rm acd}(n \Delta t;\lambda) =
  \frac{\Delta t}{\alpha}
  \frac{\Lambda_+^n U(n) - \Lambda_-^n U(n)}{\Lambda_+ - \Lambda_-}  . 
\end{equation}
where $\Lambda_\pm = (1-k/2) \pm i \sqrt{1 - (1-k/2)^2}$. In case $k=0$
\begin{equation} \label{eq:scalar_phi_acd_keq0}
  \phi_{\rm acd}(n \Delta t;\lambda) 
  = \frac{n \Delta t}{\alpha} U(n) .
\end{equation}
\end{proposition}

\begin{proof}
  The second order recursion in this case becomes
\begin{equation}
  \frac{\alpha}{\Delta t^2}
  [ u(t+\Delta t) + (-2+k(\lambda)) u(t) + u(t-\Delta t) ] = f(t)
\end{equation}
This can be written as a first order recursion
\begin{equation}
  \begin{bmatrix}
    u( n\Delta t) \\
    u((n+1)\Delta t)
  \end{bmatrix}
  =
  \Xi
  \begin{bmatrix}
    u( (n-1) \Delta t) \\
    u( n\Delta t)
  \end{bmatrix}
  +
  \frac{\Delta t^2}{\alpha}
    \begin{bmatrix}
    0 \\
    f(n \Delta t)
  \end{bmatrix}
\end{equation}
where
\begin{equation}
  \Xi =
  \begin{bmatrix}
    0 & 1 \\
    -1 & 2-k(\lambda) 
  \end{bmatrix} .
\end{equation}
It follows that
\begin{equation}
  \phi_{\rm acd}(n \Delta t,\lambda) =
  \left\{
    \begin{array}{ll}
      0 & \text{if $n \le 0$} \\
      \frac{\Delta t}{\alpha} \left[ \Xi^{n} \right]_{(1,2)}
        & \text{if $n > 0$}
    \end{array}
  \right.
\end{equation}
where, for a matrix $M$, $[ M ]_{(i,j)}$ denotes the $(i,j)$ component
of $M$.

Assuming $0 < k(\lambda) < 4$, the matrix $\Xi$ has two eigenvalues, given by,
cf.\ (\ref{eq:Lambda_pm_stability})
\begin{equation}
  \Lambda_\pm(\lambda) = (1-k(\lambda)/2) \pm 
  i \sqrt{ 1 - \left( 1-k(\lambda)/2 \right)^2}
\end{equation}
The decomposition of $\Xi$ is 
\begin{equation}
  \Xi = V \begin{bmatrix}
    \Lambda_-(\lambda) & 0 \\
    0 & \Lambda_+(\lambda)
  \end{bmatrix}
  V^{-1} , \qquad 
  V = \begin{bmatrix}
    1 & 1 \\
    \Lambda_- & \Lambda_+ .
  \end{bmatrix}
\end{equation}
Thus, for $0 < k(\lambda) < 4$,
$\left[ \Xi^{n} \right]_{(1,2)}
=
\frac{\Lambda_+^n - \Lambda_-^n }{\Lambda_+ - \Lambda_-}$, 
which implies (\ref{eq:scalar_phi_acd}).

If $k(\lambda)=0$, 
$
\left[\Xi^n \right]_{1,2} = n
$
which implies (\ref{eq:scalar_phi_acd_keq0}).
\end{proof}  

Associated with $\Lambda_+$ is the following frequency parameter
\begin{equation} \label{eq:define_Omega_of_lambda}
  \Omega(\lambda)
  = (1/\Delta t) \arg(\Lambda_+)
  = (1/\Delta t) \arccos(1-k(\lambda)/2) .
\end{equation}
This implies that $\Lambda_\pm^n = e^{ \pm i n \Delta t
  \Omega(\lambda)}$.
  
The Fourier transform $\widehat{\phi}_{\rm acd}(\tau, \lambda)$ can
now be obtained from the Fourier transform of $U(n)$ given in
Proposition~\ref{prop:FT_unit_step}
in Appendix~\ref{app:Fourier_transform_unit_step},
and the rule that modulation of a
function corresponds to translation of its Fourier transform. This leads to the
following
\begin{proposition}
  If $0 < k(\lambda) < 4$ and $\Omega(\lambda)$ is as in
  (\ref{eq:define_Omega_of_lambda}) then
  \begin{equation}
    \widehat{\phi}_{\rm acd}(\tau, \lambda)
    = \frac{1}{2i \alpha \sin (\Omega \Delta t)}
    \bigg[ \widehat{U}( \Delta t (\tau - \Omega(\lambda) ) )
    - \widehat{U}( \Delta t (\tau + \Omega(\lambda) ) ) \bigg] .
  \end{equation}
  If $k(\lambda) = 0$ then
  \begin{equation}
    \widehat{\phi}_{\rm acd}(\tau, \lambda)
    = \frac{i}{\alpha} \widehat{U}'( \Delta t \tau) .
  \end{equation}
\end{proposition}

We next state and prove a result concerning the convergence
speed. Multiple bounds are provided, each uniform in $\lambda$
over a different subsets of the domain containing $\lambda$. This is
related to special behavior that occurs for certain values of $\lambda$.
A first special case is when 
$\lambda$ is near $-\omega^2$ which corresponds to the case that
$k(\lambda)$ is near 0. In this case the factor
$\frac{1}{\sin (\Omega \Delta t)}$ becomes large, or, in the limiting
case that $k(\lambda)=0$, the formula for the $\widehat{\phi}_{\rm
  acd}$ is of a special form involving $\widehat{U}'$ instead of
$\widehat{U}$ (note that $\widehat{U}'$ is more singular).
A second special case is when $\lambda$ is near zero. In this case
$\omega$ is near $\Omega(\lambda)$. The closer $\omega$ is to the
singularity at $\Omega(\lambda)$, the narrower $\psi_{1/T}$ must be to
have a good approximation, so the larger $T$ should be. In each case
we assume $k(\lambda) < 4$ uniformly, which implies that
\begin{equation}
  \lambda < \text{constant} < \frac{4 \alpha}{\Delta t^2} - \omega^2 .
\end{equation}
The assumption that $\chi$ is $C^\infty$ is replaced by different, more
specific assumptions on $\widehat{\chi}$, so that the effect of the
precise properties of $\chi$ becomes clearer. 

\begin{theorem} \label{th:speed_of_convergence}
  Instead of the assumption that $\chi$ is $C^\infty$, assume that
  $\chi$ is a $C^0$ admissible weight function that satisfies the
  following additional assumptions:
  First there is an integer $N_1 \ge 3$ such that
\begin{align}
  \label{eq:decay_hat_chi}
  \widehat{\chi}(\sigma) = O(|\sigma|^{-N_1})  \qquad
  \widehat{\chi}'(\sigma) = O(|\sigma|^{-N_1}) \qquad
  \widehat{\chi}''(\sigma) = O(|\sigma|^{-N_1})  \qquad \sigma \to
                                  \pm \infty 
\end{align}
and secondly there is an integer $N_2 \ge 1$ such that
\begin{equation} \label{eq:assumption_N2}
  \int \left| \widehat{\chi}(\sigma)  \sigma^{N_2} \right| \, d\sigma < \infty
  \qquad \text{and the moments of $\widehat{\chi}$ of order $1 ,\ldots, N_2-1$ are zero.}
\end{equation}
Case (I): Suppose $\gamma < 0$ and $N^{\brackI} = \min(N_1+2,N_2)$. There
is a constant $C^{\brackI}$ such that
\begin{equation} \label{eq:estimate_caseI}
  | s_T^{\rm acd}(\lambda) - \lambda |
  \le
  C^{\brackI} T^{-N^{\brackI}} ,
  \qquad \text{uniformly in $\lambda \in [-\omega^2,\gamma]$}
\end{equation}
Case (II): Suppose $-\omega^2 < \gamma < 0 < \delta <
\frac{4\alpha}{\Delta t^2} - \omega^2$ and $N^{\brackII} =
\min(N_1+1,N_2)$. There is a constant $C^{\brackII}$ and a constant $\tilde{C}$ such that
\begin{equation}
  | s_T^{\rm acd}(\lambda) - \lambda |
  \le
  C^{\brackII} T^{-N^{\brackII}} |\lambda|^{-N^{\brackII}-1} ,
  \qquad \text{uniformly in $\lambda \in [\gamma, 0) \cup (0, \delta]$} ,
\end{equation}
assuming $T > \tilde{C} /|\lambda|$.
\end{theorem}

In the remainder of this subsection we will prove this result. In
subsection~\ref{subsec:discuss_convergence_speed_result} we will
discuss the result. 

\begin{lemma} \label{lem:estimate_chi_Tnu}
  Suppose $\chi$ satisfies (\ref{eq:decay_hat_chi}). Let
  $g_j(\nu)$, $\nu \in \RR$ be defined by
  $g_j(\nu) = \widehat{\chi}(T (\nu + j 2 \pi))$. Suppose $\pm \nu$ is in an interval
  $[ C_1 \delta, C_2 \delta]$, $0< C_1<C_2$, where $\delta$ is in
  $(0,1]$ and denote by
  $g_j^{(k)}$ the $k$-th derivative, then there is a constant $C$ such
  that for $k=0,1,2$, for all $\delta \in (0,1]$ and
  $\nu$ such that $\pm \nu \in [ C_1 \delta, C_2 \delta]$
  \begin{equation}
    \begin{aligned}
      | g_0^{(k)}(T \nu) | \le {}& C T^{-N_1+k} \delta^{-N_1}
      && 
      \\
      | g_j^{(k)}(T \nu) | \le {}& C T^{-N_1+k} |j|^{-N_1}
      &&
      \text{if $j \neq 0$} .
    \end{aligned}
  \end{equation}
\end{lemma}
    
\begin{proof}
This follows straightforwardly from (\ref{eq:decay_hat_chi}) by differentiation.
\end{proof}

\begin{proof}[Proof of Theorem~\ref{th:speed_of_convergence}]
By rescaling problem quantities, we may assume that $\Delta
t=1$. In terms of original problem quantities this amounts to 
a change of variables
\begin{equation}
  \nu = \Delta t \tau .
\end{equation}
We will also write
\begin{equation}
  \nu_1 = \Delta t \Omega(\lambda) , \qquad 
  \nu_2 = \Delta t \omega .
\end{equation}
In this notation, case (I) is the case that $0 \le \nu_1 <
\text{constant} < \nu_2 $, and case (II) is the case that
$0 < \text{constant} < \nu_1 < \text{constant} < \pi$.
For brevity we will also omit the subscript from
$\widehat{\phi}_{\rm acd}$. It is given by  
\begin{equation} \label{eq:widehat_phi}
  \widehat{\phi}(\nu) =
  \left\{
    \begin{array}{ll}
      \displaystyle\frac{i}{2 \alpha \sin(\nu_1)} \left[ \widehat{U}(\nu +\nu_1) -
      \widehat{U}(\nu -\nu_1) \right]
      &
        \text{if $k(\lambda)\neq 0$,}
      \\
      \displaystyle\frac{i}{\alpha} \widehat{U}'(\nu)
      & \text{if $k(\lambda) = 0$.}
    \end{array}
  \right.
\end{equation}
The quantity $s_T^{\rm acd}(\lambda)$ satisfies, where we omit the superscript,
\begin{equation}
  s_T(\lambda) = \int_{\TT_{2\pi}} \psi_{1/T}(\nu_2 - \nu) \widehat{\phi}(\nu) \, d\nu .
\end{equation}

We first consider case (I).
To be shown is that
\begin{equation} \label{eq:s_TOexpressionI}
  s_T(\lambda) = \widehat{\phi}(\nu_2) + O(T^{-N^{\brackI}}) , \qquad T
  \to \infty.
\end{equation}
In this case, to handle the factor
$\frac{1}{\nu_1}$ around $\nu_1=0$, we write $\widehat{\phi}(\nu)$ as
\begin{equation} \label{eq:phi_integral_phiprime}
  \widehat{\phi}(\nu) =
  \frac{i \nu_1}{2\alpha \sin(\nu_1)}
  \int_{-1}^1 \widehat{U}'(\nu+s \nu_1) \, ds ,
\end{equation}
which is valid for $k(\lambda)=0$ and for $k(\lambda) \neq 0$.
The integral over $s$ can be substituted on
both sides of the equality sign in (\ref{eq:s_TOexpressionI}) and it
is sufficient to obtain an estimate for the integrand. Let
$\tilde{\nu} = s \nu_1$, $s \in[-1,1]$. It is
sufficient to show that
\begin{equation} \label{eq:estimate_using_widehatUprime}
  \int_{\TT_{2\pi}} \psi_{1/T}(\nu_2 - \nu) \widehat{U}'(\nu+\tilde{\nu}) \,
  d\nu
  =
  \widehat{U}'(\nu_2+\tilde{\nu})  + O(T^{-N^{\brackI}}) , \qquad T  \to \infty 
\end{equation}
uniformly in $\tilde{\nu}$.  We insert the expression for
$\widehat{U}$ from Proposition~\ref{prop:FT_unit_step}, and separately
consider the contributions from the delta function and the cotangent
function. Let
\begin{equation}
  \begin{aligned}
    I_1 = \int_{\TT_{2\pi}} \psi_{1/T}(\nu_2 - \nu) \delta'(\nu+\tilde{\nu}) \, d\nu
    \qquad \text{and} \qquad
    I_2 = \int_{\TT_{2\pi}} \psi_{1/T}(\nu_2 - \nu)
    \cot'(\frac{1}{2}(\nu+\tilde{\nu})) \, d\nu .
  \end{aligned}
\end{equation}
then the left hand side of integral
(\ref{eq:estimate_using_widehatUprime}) equals $\pi I_1 -
\frac{i}{2} I_2$,
and it is sufficient to show that
\begin{equation} \label{eq:estimate_I1_I2}
  I_1 = O(T^{-N^{\brackI}})
  \qquad \text{and} \qquad
  I_2 = \cot'(\frac{1}{2}(\nu+\tilde{\nu})) + O(T^{-N^{\brackI}}) ,
\end{equation}
$T \to \infty$.

Using the rules for differentiation and integration of delta
functions and Lemma~\ref{lem:psi_hatchi}, $I_1$ can be written as
\begin{equation}
  I_1 = - \psi_{1/T}'(\nu_2 + \tilde{\nu}) 
= - \frac{\tilde{T}}{2\pi} \sum_{j \in \ZZ}
  \widehat{\chi}(\tilde{T} (\nu_2 + \tilde{\nu}+j 2\pi)) .
\end{equation}
The first part of (\ref{eq:estimate_I1_I2}) now follows from
Lemma~\ref{lem:estimate_chi_Tnu}.

Using Lemma~\ref{lem:psi_hatchi} and a periodic cutoff function
as in the proof of Theorem~\ref{th:convergence_acd}, the integral
$I_2$ is rewritten as
\begin{equation}
  I_2 =  \frac{\tilde{T}}{2\pi} \sum_{j \in \ZZ}
  \int_\RR \widehat{\chi}(\tilde{T} (\nu_2 - \nu+j 2\pi))
  \theta_{\rm per}(\nu_2 - \nu ) \cot'(\nu+\tilde{\nu}) \, d\nu 
\end{equation}
We split the integration domain in three
regions for $j \neq 0$ and in four regions for $j=0$.
In both cases, region 2 is of width $\sim 1/T$ around $\tilde{\nu}$, the
associated cutoff function $\theta_2$ is equal to 1 one a region of
this width, and vanishes outside a large region of this width.
Regions 1 and 3 are below and above region 2 respectively.
In case of $j=0$, an region of width $\sim 1$ around $\nu = \nu_2$,
staying $\sim 1$ away from $-\tilde{\nu}$ is split off from region 3, with a
cutoff function $\theta_4$. Will absorb $\theta_{\rm per}$ in these
cutoff functions.
For region and cutoff function $k$ and summation index $j$ there are then the
contributions
\begin{equation}
  I_{2,j,k} = 
  \frac{\tilde{T}}{2\pi} \int_\RR \widehat{\chi}(\tilde{T} (\nu_2 - \nu))
  \theta_k(\nu +j 2\pi)
  \cot'(\frac{1}{2} (\nu+\tilde{\nu})) \, d\nu .
\end{equation}
and $I_2 = I_{2,0,4} + \sum_{j \in \ZZ} \sum_{k=1}^3 I_{2,j,k}$.
We claim that for regions $k=1,2,3$
\begin{equation} \label{eq:error_caseI}
  I_{2,j,k} = O(T^{-N_1+2}  (1+|j|)^{-N_1}) .
\end{equation}
For region 2 we observe that the action of $\cot'$ on a test function
supported on $[-w,w]$ can be bounded by 
\begin{equation}
  | \langle \cot', \phi \rangle |
  = | \langle \cot, \phi' \rangle |
  \le C |w| \, \| \phi'' \|_{L^\infty([-w,w])} 
\end{equation}
The estimate (\ref{eq:error_caseI}) follows from
this fact and the assumptions on $\widehat{\chi}$.
For region 3 there is a contribution of the form
\begin{equation}
  I_{2,j,3} = \frac{\tilde{T}}{2\pi}
  \int \cot'(\frac{1}{2}(\nu+\tilde{\nu}))
  \widehat{\chi}(\tilde{T} (\nu_2-\nu + 2\pi j)) \theta_3(\nu) \,  d
  \nu .
\end{equation}
where the integral stays $\sim 1/T$ away from the singularity of $\cot'$
and, for $j=0$ it stays $\sim 1$ away from $\nu = \nu_2$ where
$\widehat{\chi}$ is large. By these properties and Lemma~\ref{lem:estimate_chi_Tnu}
equation (\ref{eq:error_caseI}) is true for region 3. In a
similar way it is true for region 1. 

We next consider the contribution for case (I) around $\nu = \nu_2$. 
To show this we need assumption (\ref{eq:assumption_N2}) and 
extend the reasoning around (\ref{eq:expand_phi_around_omega})
to higher order. Let $g(\nu) =\cot'(\frac{1}{2}(\nu+\tilde{\nu}))
  \theta_4(\nu)$. This function can be expanded
\begin{equation}\label{eq:estimate_expand_order_N2}
  g(\nu)
  =
  g(\nu_2)
  + g'(\nu_2) (\nu-\nu_2)
  + \ldots 
  + \frac{1}{(N_2-1)!}g^{(N_2-1)}(\nu_2) (\nu-\nu_2)^{N_2-1}
  + h_{N_2}(\nu) (\nu-\nu_2)^{N_2} ,
\end{equation}
where $h_{N_2}$ is bounded by a constant times the
$N_2$-th derivative of $g$.
Then the relevant integral is transformed
\begin{equation} \label{eq:estimate_transform_order_N2}
  \begin{aligned}
  {}& \frac{\tilde{T}}{2\pi} \int g(\nu) \widehat{\chi}(\tilde{T} (\nu_2 - \nu)) \, d\nu
= 
  \frac{g(\nu_2)}{2\pi}
  \int  \widehat{\chi}(x) \, dx
  + \frac{g'(\nu_2)}{2\pi \tilde{T}}
  \int \widehat{\chi}(x) x \, dx
\\
  {}& + \ldots
  + \frac{g^{(N_2-1)}(\nu_2)}{2\pi \tilde{T}^{N_2-1} (N_2-1)!}
  \int \widehat{\chi}(x) x^{N_2-1} \, dx
+ \frac{1}{2\pi \tilde{T}^{N_2}}
  \int \widehat{\chi}(x) x^{N_2} h_{N_2}(x/\tilde{T}) \, dx .
\end{aligned}
\end{equation}
with $h_{N_2}$ bounded by a bound for  
$\frac{d^{N_2} g}{d\nu^{N_2}}(\nu)$.
By assumption (\ref{eq:assumption_N2}) we find that
\begin{equation}
  I_{2,0,4} = g(\nu_2) + O(T^{-N_2})
  = \cot'(\frac{1}{2} (\nu_2 + \tilde{\nu})) + O(T^{-N_2}) .
\end{equation}
This concludes the proof of (\ref{eq:estimate_caseI}).

We next consider case (II). In this case we assume $T \gtrsim
1 / |\lambda|$. It is sufficient to show that
\begin{equation} 
  \int_{\TT_{2\pi}} \psi_{1/T}(\nu_2 - \nu) \widehat{U}(\nu \pm \nu_1) \,
  d\nu
  =
  \widehat{U}(\nu_2 \pm \nu_1)  + O(T^{-N^{\brackII}} |\lambda|^{-N^{\brackII-1}}) ,
    \qquad T  \to \infty .
\end{equation}
We note that $|\nu_2 - \nu_1| \sim |\lambda|$, while
$\nu_2 + \nu_1 \sim 1$. For the $+$ sign, in fact $O(T^{-N^{\brackII}})$
can be obtained. We will treat the $-$ sign, the $+$ sign can be done
in the same way with minor modifications.
Using the expression for $\widehat{U}$ in
Proposition~\ref{prop:FT_unit_step} the integral on the left can be
written as a sum $\pi I_3 - \frac{i}{2} I_4$ where $I_3$ and $I_4$ are
defined by
\begin{equation}
  \begin{aligned}
    I_3 = \int_{\TT_{2\pi}} \psi_{1/T}(\nu_2 - \nu) \delta(\nu-\nu_1) \, d\nu
    \qquad \text{and} \qquad
    I_4 = \int_{\TT_{2\pi}} \psi_{1/T}(\nu_2 - \nu)
    \cot(\frac{1}{2}(\nu-\nu_1)) \, d\nu .
  \end{aligned}
\end{equation}
and it is sufficient to show that
\begin{equation} \label{eq:estimate_I3_I4}
  I_3 = O(T^{-N^{\brackII}} |\lambda|^{-N^{\brackII} - 1})
  \qquad \text{and} \qquad
  I_4 = \cot(\frac{1}{2}(\nu-\nu_1)) + O(T^{-N^{\brackII}}  |\lambda|^{-N^{\brackII} - 1}) ,
\end{equation}
$T \to \infty$.

The integral $I_3$ equals
\begin{equation}
  I_3 = \psi_{1/T}(\nu_2 - \nu_1) 
  = \frac{\tilde{T}}{2\pi} \sum_{j \in \ZZ}
  \widehat{\chi}(\tilde{T} (\nu_2 -\nu_1) ) .
\end{equation}
By Lemma~\ref{lem:estimate_chi_Tnu} the first part of (\ref{eq:estimate_I3_I4})
follows.

The integral $I_4$ can be written as
\begin{equation}
  I_4 =  \frac{\tilde{T}}{2\pi} \sum_{j \in \ZZ}
  \int_\RR \widehat{\chi}(\tilde{T} (\nu_2 - \nu+j 2\pi))
  \theta_{\rm per}(\nu_2 - \nu ) \cot(\nu-\nu_1) \, d\nu 
\end{equation}
We split the integration domain in three
regions for $j \neq 0$ and in four regions for $j=0$.
In both cases, region 2 is of width $\sim 1/T$ around $\tilde{\nu}$, the
associated cutoff function $\theta_2$ is equal to 1 one a region of
this width, and vanishes outside a large region of this width.
Regions 1 and 3 are below and above region 2 respectively.
In case of $j=0$, a region of with $\sim |\lambda|$ around $\nu = \nu_2$,
staying $\sim |\lambda|$ away from $\nu_1$ is split off from region 1
or region 3, with a
cutoff function $\theta_4$. Will absorb $\theta_{\rm per}$ in these
cutoff functions.
For region and cutoff function $k$ and summation index $j$ there are then the
contributions
\begin{equation}
  I_{4,j,k} = 
  \frac{\tilde{T}}{2\pi} \int_\RR \widehat{\chi}(\tilde{T} (\nu_2 - \nu))
  \theta_k(\nu +j 2\pi)
  \cot(\frac{1}{2} (\nu - \nu_1)) \, d\nu .
\end{equation}
and $I_4 = I_{4,0,4} + \sum_{j \in \ZZ} \sum_{k=1}^3 I_{4,j,k}$.
We claim that for regions $k=1,2,3$
\begin{equation} \label{eq:error_caseII}
  I_{4,0,k} = O(T^{-N_1+1} |\lambda|^{-N_1})
  \qquad \text{and} \qquad
  I_{4,j,k} = O( T^{-N_1+1} |j|^{-N_1} ) ,\qquad \text{if $j \neq 0$}.
\end{equation}
For region 2 we observe that the action of $\cot$ on a test function
supported on $[-w,w]$ can be bounded by 
\begin{equation}
  | \langle \cot, \phi \rangle |
  \le C |w| \, \| \phi' \|_{L^\infty([-w,w])} 
\end{equation}
The estimate (\ref{eq:error_caseII}) follows from
this fact and the assumptions on $\widehat{\chi}$.
For region 3 there is a contribution of the form
\begin{equation}
  I_{2,j,3} = \frac{\tilde{T}}{2\pi}
  \int \cot(\frac{1}{2}(\nu-\nu_1))
  \widehat{\chi}(\tilde{T} (\nu_2-\nu + 2\pi j)) \theta_3(\nu) \,  d
  \nu .
\end{equation}
The support of the integrand stays $\sim 1/T$ away from the singularity
of the cotangent function and, in case $j = 0$ it stays $\sim |\lambda|$
away from the point $\nu = \nu_2$ which is where $\widehat{\chi}$ is
large. This implies  (\ref{eq:error_caseII}) for region 3. Region 1 is
treated similarly.
The treatment in equations
(\ref{eq:estimate_expand_order_N2}) and
(\ref{eq:estimate_transform_order_N2})
is used with minor modifications to show that
\begin{equation}
  I_{4,0,k} = \cot(\frac{1}{2}(\nu_2 - \nu_1)) + O(T^{-N_2}
  |\lambda|^{-N_2-1}) .
\end{equation}
This concludes the proof of the theorem.
\end{proof}

\subsection{Discussion}
\label{subsec:discuss_convergence_speed_result}

We briefly discuss 
Theorem~\ref{th:speed_of_convergence} and its assumptions.

By elementary properties of Fourier transforms
assumption (\ref{eq:decay_hat_chi}) is directly related to the
smoothness of $\chi$. For example, for a piecewise $C^\infty$
function, with finite jumps in the $k$-th derivatives, such an estimate
holds with $N_1 = k+1$.
The $k$-th moment of the Fourier transform $\widehat{\chi}$ is
proportional to $\chi^{(k)}(0)$. The second requirement
of (\ref{eq:assumption_N2}) can hence be satisfied by assuming
$\chi =1 $ on neighborhood of 0.
The first part of (\ref{eq:assumption_N2}) is again related to the
smoothness of $\chi$.

It is clear that the inversion problem becomes more difficult if the
smallest absolute eigenvalue of $H$ is closer to zero. Define the
notion of ``gap'' by
\begin{equation}
  \operatorname{gap} = \min \{ |\lambda_1|, \ldots, |\lambda_N| \} .
\end{equation}

When the matrix $S_T^{\rm acd}$ is applied as a preconditioner, the
convergence factor is of interest. In the context of
subsection~\ref{subsec:speed_of_convergence} this is given by
\begin{equation}
  \rho_{\rm c} = \max  \{ s_T^{\rm acd} \lambda  -1  \, ; \, j =1\ldots,N \}
\end{equation}
The theorem leads to the following result. Suppose one wants to obtain
a convergence factor $\rho_c < R_c$, with $R_c$ some fixed constant
$<1$. There is a constant $C$ such that one can set
\begin{equation}
  T = C / \operatorname{gap}
\end{equation}
to obtain this.
With this choice a preconditioned iterative method will converge linearly,
and the {\em total number of time steps simulated is proportional to
  $1/\operatorname{gap}$}. For the error as a function of the number
of time steps simulated one obtains a bound
\begin{equation} \label{eq:linear_convergence}
  \text{error} < C e^{ - (\text{\#timesteps})  \cdot  \operatorname{gap} }
\end{equation} 

For the WaveHoltz method of
\cite{appelo2019waveholtz} a convergence rate of
$1 - O(\operatorname{gap}^2)$ was established, see Theorem 2.3 of
\cite{appelo2019waveholtz}.  This leads to a bound
for the errors
\begin{equation}
  \text{error} < C e^{ - (\text{\#timesteps}) \cdot
    \operatorname{gap}^2 } ,
\end{equation}
which is clearly less attractive than the bound (\ref{eq:linear_convergence}).

\section{Numerical examples}
\label{sec:numerical_examples}

In this section we study some examples. The examples all involve an
optimized finite-difference discretization described in
subsection~\ref{subsec:optimized-fd} in two and three dimensions.

To study the 2-D method, a simple implementation in Julia was made.
In the 3-D case the time stepping was done in C, the computation of coefficients and
the GMRES iteration were done in Julia. In the C implementation, for
each grid point and timestep, 14 real coefficients had to be read, of
which one also had to be written back to. The C implementation used AVX
extensions but was otherwise straightforward. In most computations
double precision numbers were used. The exception was the
time-stepping performed in C. We found that the time-domain preconditioner
could also be run in single precision, since the GMRES solver would
automatically take care of rounding errors in subsequent iterations.

\subsection{2-D examples}
\label{subsec:numerical_examples_2D}

We considered three 2-D examples:
A constant velocity model defined on
the unit square referred to as vel1 and two
piecewise constant models also defined on the unit square. The first of the
two piecewise constant models allows for resonances in the circular
region.  On the exterior of the model damping layers of thickness 32
gridpoints were added to simulate an unbounded domain.
The simulations in these examples
were done using a minimum of 6 points per wavelength. In each case
this resulted in 8 timesteps per period.
See
figure~\ref{fig:model_and_sols2D} for non-constant velocity models and
some solutions in these models. 

The first objective was to the study the convergence of the
approximate solution $S_T^{\rm acd} F$ to $H^{-1} F$. For this we took models of size
$320\times320$ and $640 \times 640$ (excluding damping layers), and
let the right-hand side be a point source.
\ReviseBlue{For the constant model the
  parameter $k = \frac{\omega_{\rm ph}}{c}$ equalled $k=335.1$ for
  size $320\times 320$ and $k=670.1$ for size $640 \times 640$.} 
The ``taper'' parameter $\rho$ was varied, we considered values in
$\{ 0, 0.25, 0.5, 0.75 \}$. The value $0$ was not exactly zero, in
this case the time-harmonic right-hand side grew from 0 to 1 over one
period. 
The convergence (relative difference between approximate and true solutions)
is given in Figure~\ref{fig:large_time_convergence}. Note the
differing axes in the Figure~\ref{fig:large_time_convergence}(b).
The main conclusions from these figures are that the convergence in
the resonant model is relatively poor, and that the choice of $\rho = 0$ (no
tapering) leads to poor convergence. The detailed amount of tapering
is less important. The lower value ($\rho = 0.25$) appears to
perform best unless extremely small relative errors are required.
\newcommand{\addsubfigurelabel}[3]{%
  \begin{minipage}{#2}
    \begin{center}
      #1\\
      #3
    \end{center}
  \end{minipage}}

\begin{figure} 
  \begin{center}
    \addsubfigurelabel{(a)}{59mm}{%
      \includegraphics[width=59mm]{\figdir 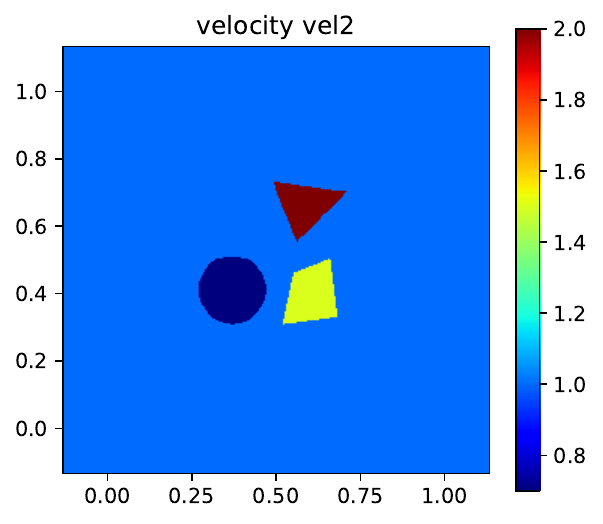}}
    \addsubfigurelabel{(b)}{59mm}{
     \includegraphics[width=59mm]{\figdir 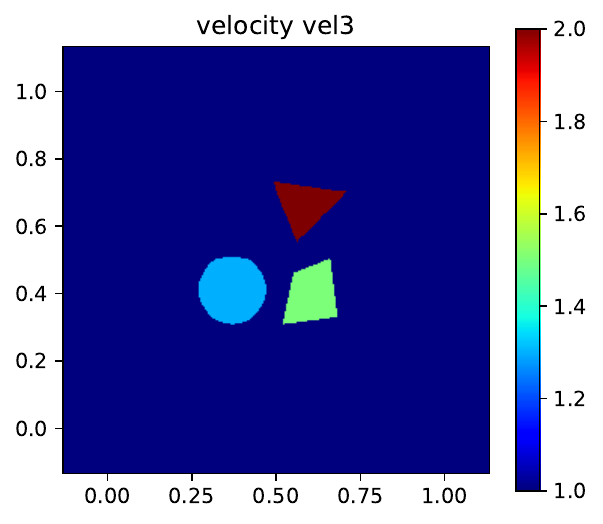}}
    \addsubfigurelabel{(c)}{59mm}{%
      \includegraphics[width=59mm]{\figdir 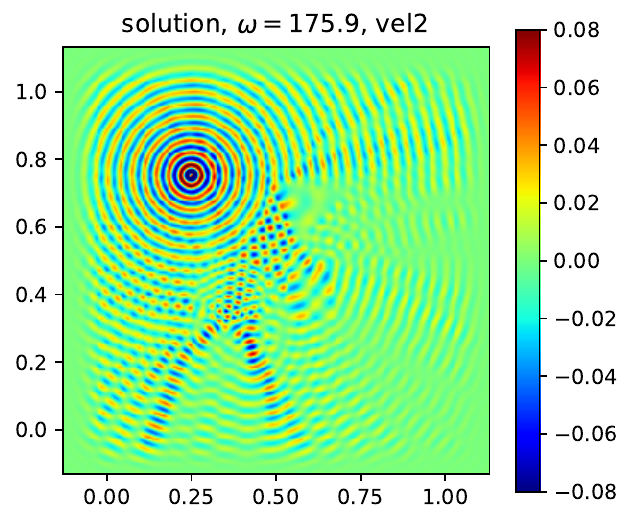}}
    \addsubfigurelabel{(d)}{59mm}{
     \includegraphics[width=59mm]{\figdir 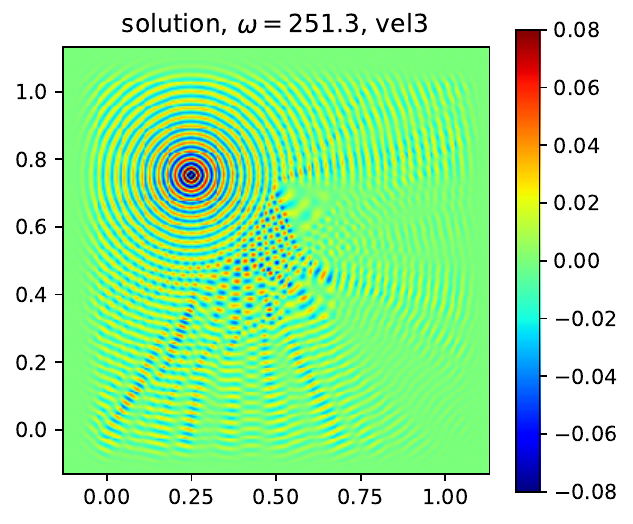}}
  \end{center}
  \caption{Two velocity models (a,b) and solutions in these models
    (c,d).
  A relatively small example of $240 \times 240$ was used so that the
  wavefield oscillations are still visible.
}\label{fig:model_and_sols2D}
\end{figure}

\begin{figure} 
  \begin{center}
    \addsubfigurelabel{(a)}{63mm}{%
    \includegraphics[width=63mm]{\figdir 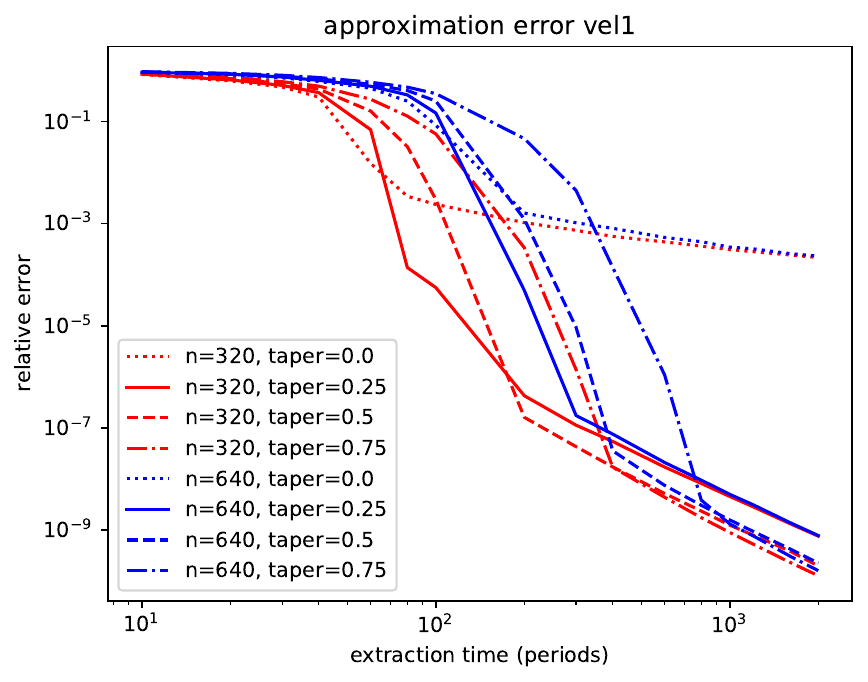}} \hspace*{5mm}
    \addsubfigurelabel{(b)}{63mm}{%
    \includegraphics[width=63mm]{\figdir 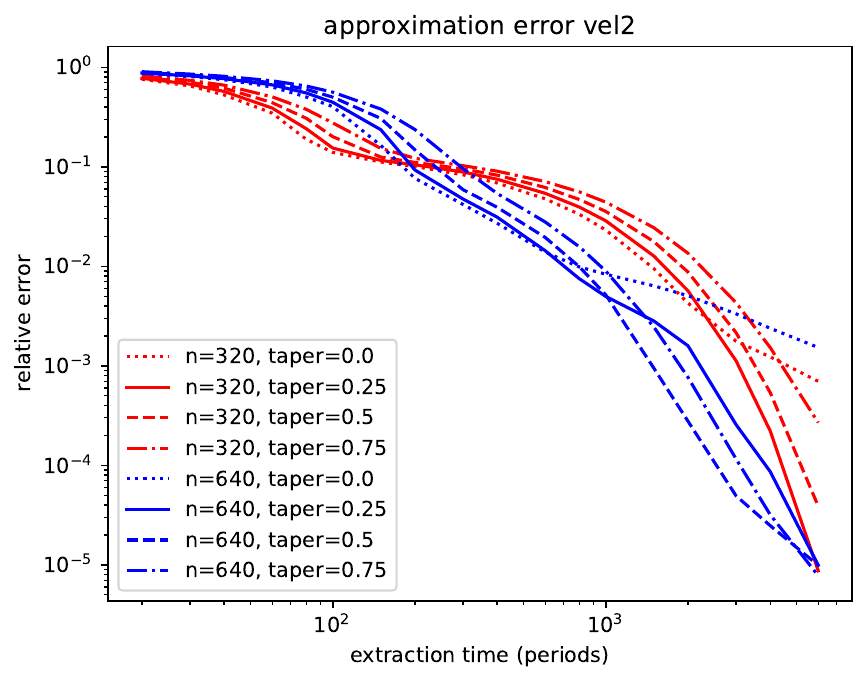}}\\
    \addsubfigurelabel{(c)}{63mm}{%
    \includegraphics[width=63mm]{\figdir 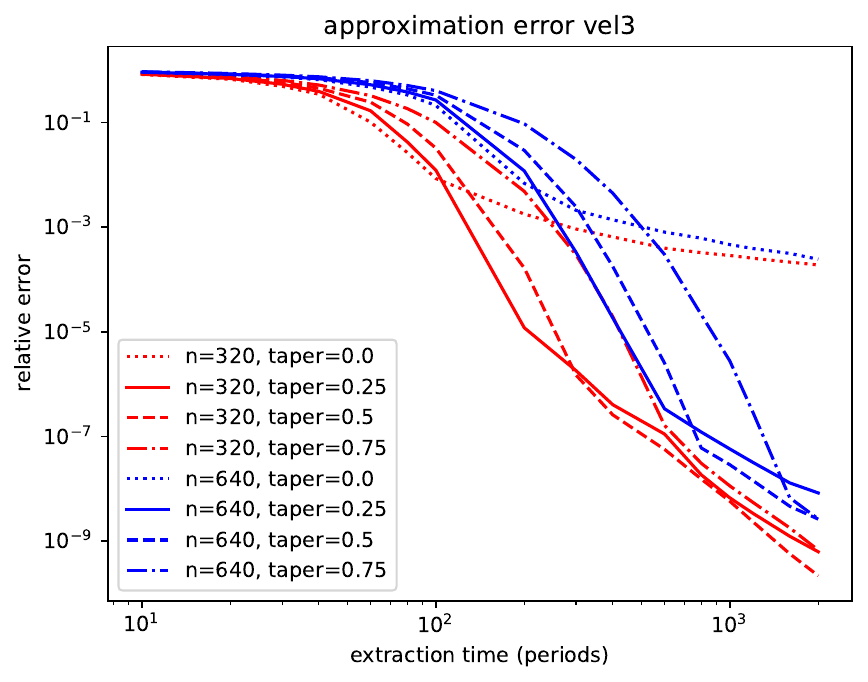}}\\
  \end{center}
  \caption{Convergence behavior of large time approximate
    solutions. Note the different scales in figure (b).}
  \label{fig:large_time_convergence}
\end{figure}

GMRES convergence for the three velocity models of size 640 gridpoints
(excluding damping layers) are given in
Figure~\ref{fig:GMRES_convergence}. In each case the number of periods
of the preconditioner was varied.
In Figures~\ref{fig:GMRES_convergence}(a-c) the taper parameter
was 0.25
and both the error indicator from GMRES and the true error are
plotted. In Figure~\ref{fig:GMRES_convergence}(d) the GMRES error was
plotted and the taper parameter was varied.
On the x axis, the number of iterations times the number of periods per
iteration was displayed. This is roughly, but not quite proportional
to the cost, the main difference being that one more preconditioner
application is needed for the right-hand side in the preconditioned
system.
In velocity models 1 and 3, there was no speedup compared to the
direct application of the time-domain solver. Neither was the method
much slower. However, in the resonant velocity model (velocity model
2), the performance using GMRES was substantially better than when the
time-domain solver was used directly.

In the examples about 50 to 100 periods for the preconditioner worked
best. This is approximately the size of the example in wavelenghts. 
\begin{figure} 
  \begin{center}
    \addsubfigurelabel{(a)}{63mm}{%
    \includegraphics[width=63mm]{\figdir 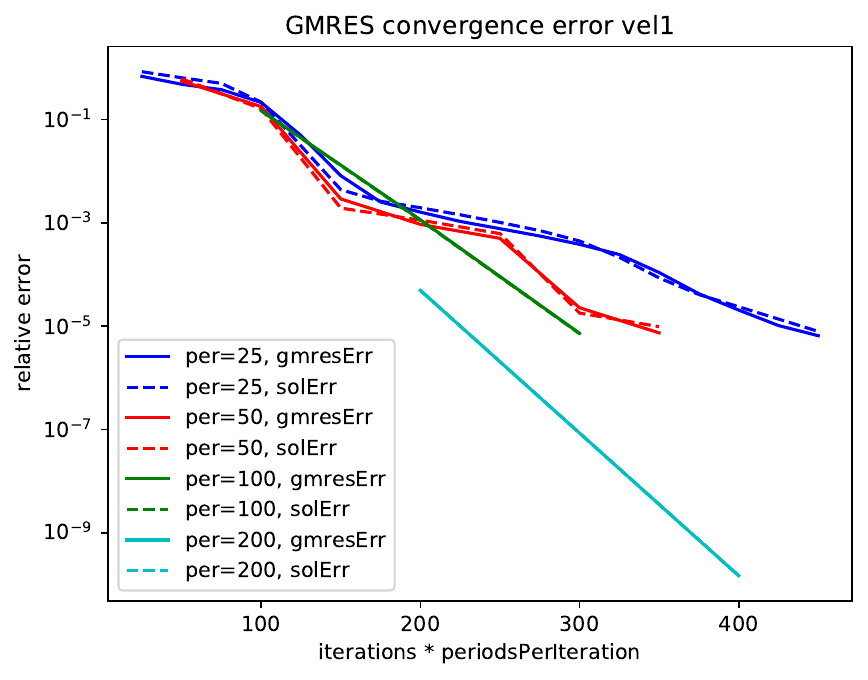}} \hspace*{5mm}
    \addsubfigurelabel{(b)}{63mm}{%
    \includegraphics[width=63mm]{\figdir 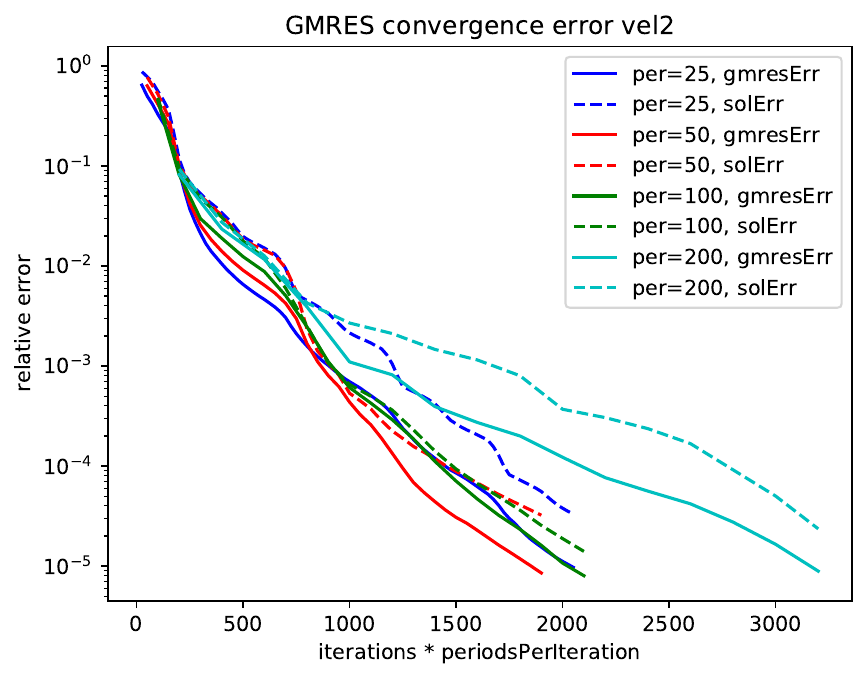}}\\
    \addsubfigurelabel{(c)}{63mm}{%
    \includegraphics[width=63mm]{\figdir 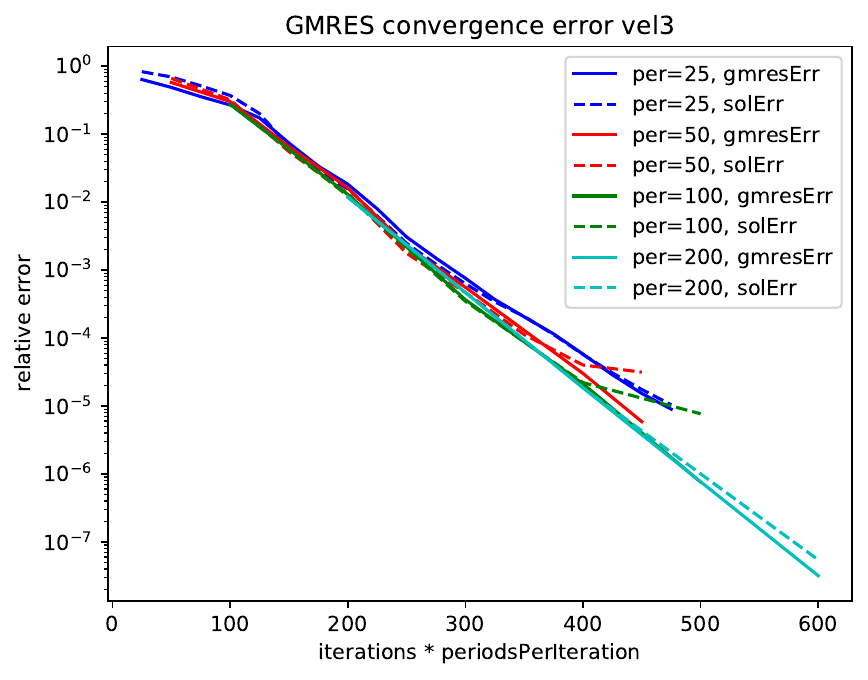}} \hspace*{5mm}
    \addsubfigurelabel{(d)}{63mm}{%
    \includegraphics[width=63mm]{\figdir 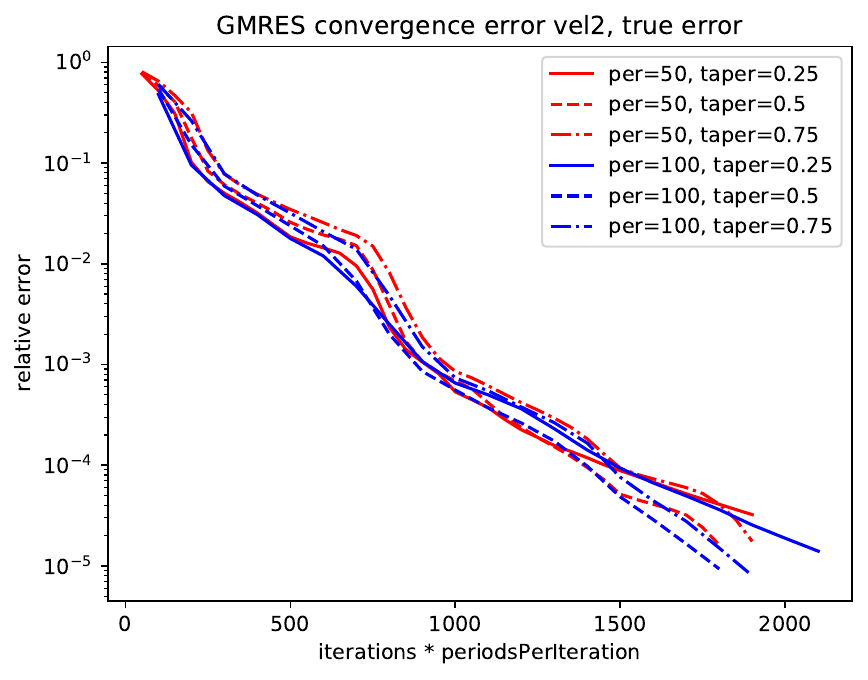}}
  \end{center}
\caption{Preconditioned GMRES convergence for different models and different
  parameter choices} \label{fig:GMRES_convergence}
\end{figure}

It is confirmed here that, in case of resonances, long time simulation
(i.e.\ applying the limiting amplitude principle directly) leads to
poor convergence.  For our time-domain
preconditioner, a substantial speedup can be obtained by using Krylov
accelleration and a smooth window function. Exact controllability is
another way of obtaining a speedup, cf. \cite[Fig.\
7]{grote2020parallel}. How these two methods compare is difficult to
say based on what is published.  Poor convergence in case of
resonances has also been observed in examples using WaveHoltz
\cite[Fig 4.12(b)]{appelo2019waveholtz} and other iterative methods.

\subsection{Examples in 3-D}
\label{subsec:numerical_examples_3D}
We also studied the performance of the algorithm in three
dimensions. In this case, the velocity model was the SEG/EAGE Salt
Model%
\ReviseBlue{\footnote{%
    \ReviseBlue{see \url{https://wiki.seg.org/wiki/SEG/EAGE_Salt_and_Overthrust_Models}}}%
  }.
In this case our goal was to obtain a first estimate of the
computational cost of the method. A minimum of 6 gridpoints per
wavelength was used. The Julia/C code was run on a 2019 MacBook Pro with
a 2.6 GHz 6-core Intel i7 processor and 16 GB main memory. A GMRES
error reduction with a factor  $10^{-5}$ was required.
The results are in Table~\ref{tab:seg_eage_salt_example}.

\begin{table} \label{tab:seg_eage_salt_example}
\begin{center}
  \begin{tabular}{| l | c | c | c |}
    \hline
    physical frequency & 2.5 Hz & 4 Hz & 6 Hz\\
    problem size   & $200 \times 200 \times 106$
                                & $280 \times 280 \times 130$
    & $392 \times 392 \times 165$\\
    degrees of freedom         & 4.24 e6 & 1.02 e7 & 2.54 e7\\
    time steps/period   & 8 & 8 & 8 \\
    periods/iteration     & 25 & 40 & 60 \\
    taper parameter          & 0.25 & 0.25 & 0.25\\
    iterations                 & 6 & 6 & 6 \\ 
    computation time & 45 s & 166 s & 601 s\\ \hline
  \end{tabular}
\end{center}
\caption{Computational results for the SEG/EAGE Salt model}
\end{table}

\subsection{Comparison with diagonally preconditioned GMRES}
\label{subsec:diagonally_preconditioned_gmres}

\begin{table} \label{tab:diag_prec_gmres}
\begin{center}
  \begin{tabular}{| l | c || c | c || c | c | }\hline
    \multicolumn{2}{|c}{velocity model}
    & \multicolumn{2}{|c}{diagonal precon}
    & \multicolumn{2}{|c|}{time-domain precon} \\ \hline
    name & size
    & \#iterations (time) & solerr & \#timesteps (time) & solerr \\ \hline
    vel1 & $320 \times 320$
    & 938 (8.8 s) & 5.1e-5 & 2000 (6.3 s) & 3.0e-6 \\ 
    vel1 & $640 \times 640$    
    & 1761 (61 s) & 7.1e-5 & 3200 (38 s) & 9.7e-6 \\ \hline
    vel2 & $320 \times 320$
    & 20503 (204 s) & 9.7e-4 & 9600 (33 s) & 1.9e-5 \\
    vel2 & $640 \times 640$    
    & 56296 (2066 s) & 1.8e-4 & 15600 (201 s) & 3.1e-5 \\ \hline 
    vel3 & $320 \times 320$
    & 1613 (19 s) & 3.3e-5 & 2400 (7.5 s) & 4.0e-5 \\
    vel3 & $640 \times 640$    
    & 3197 (110 s) & 4.8e-5 & 3600 (51 s) & 3.8e-5 \\ \hline
  \end{tabular}
\caption{Comparison of time-domain preconditioner with a diagonal preconditioner}
\end{center}
\end{table}
In Table~\ref{tab:diag_prec_gmres} a comparison between the
time-domain preconditioner and diagonally preconditioned GMRES is
made. In both cases the number of matrix applications, the computation
time and the error compared to the true solution is displayed.
The computations were done in Julia.

Diagonally preconditioned GMRES is consistently slower. This is
because per matrix application, the diagonally preconditioned GMRES
method involves more work, since the vectors spanning the Krylov
subspace must be manipulated. 
In specific scenarios there are more advantages of the time-domain
preconditioner. When GPUs are used it is an
advantage that the Krylov subspace is manipulated infrequently, which
means that the Krylov subspace doesn't have to be stored in the scarce device
memory. In inner-outer iterative methods the time-domain
preconditioner could be used as an inner method, with the advantage
that it is linear.


\section{Explicit scheme for non-diagonal damping}
\label{sec:generalize}


If the matrix $B$ in
(\ref{eq:central_differences_damped_leapfrog_verlet}) is non-diagonal,
the associated time-stepping method becomes implicit, which is not
practical for a hyperbolic time-dependent system. Here we introduce
a time-integrator that is explicit even in the case of
non-diagonal damping by following 
similar steps as in section~\ref{sec:method}. Possible
applications are time-harmonic PDE's discretized using standard finite
elements and finite difference discretizations with PML boundary layers.
The associated time integrator has stricter CFL conditions and is
therefore only proposed for the case
that $B$ is non-diagonal.

\subsection{Definition of the method}
\label{subsec:definition_bd}

The starting point of our {\em frequency adapted backward differences
  damped leapfrog} scheme is a straightforward modification of
(\ref{eq:central_differences_damped_leapfrog_verlet}) in which
backward differences are used to discretize the damping term
\begin{equation} \label{eq:backward_differences_damped_leapfrog_verlet}
  \frac{1}{\Delta t^2} \left( u_{n+1} - 2 u_n + u_{n-1}
  \right)
  + \frac{B}{\Delta t} \left( u_n - u_{n-1} \right) 
  + A u_n = f_n .
\end{equation}
The associated time integrator will be called backward differences
damped leapfrog.
\begin{definition} \label{def:Ibd}
  Let $K, L$ and $g_n$ be as in (\ref{eq:K_L_central_differences}).
  {\em Backward differences damped leapfrog} will be defined as the time
  integrator given by
  \begin{equation} \label{eq:bd_integrator}
    u_{n+1}
    = I_{\rm bd}(u_n,u_{n-1},f_n) 
    := (2 - K - L) u_n - (I - L) u_{n-1} + g_n  .
  \end{equation}
\end{definition}

By choosing $K$, $L$ differently, the resulting time integrator can be
made exact for time-harmonic signals of frequency $\omega$. This is
shown in the following Proposition, which is an equivalent of  
Proposition~\ref{prop:reproduce_symbol_at_omega}. The associated time
integration method will be denoted by $I_{\rm abd}$.

\begin{proposition} \label{prop:reproduce_symbol_at_omega_bd}
  Let $u_n$ and $f_n$ be related to $U,F \in \CC^N$ by
  \begin{equation}
    u_n = e^{i \omega n \Delta t} U , \qquad
    f_n = e^{i \omega n \Delta t} F
  \end{equation}
  and let $\alpha$ and $\beta$ be as defined in (\ref{eq:define_alpha_beta}).
  Then $U,F$ satisfy (\ref{eq:time-harmonic_from_semi-discrete}) if
  and only if $u_n$, $f_n$ satisfy
  \begin{equation} \label{eq:mod_backward_differences_damped_leapfrog_verlet}
    \frac{1}{\Delta t^2} \left( u_{n+1} - 2 u_n + u_{n-1}
    \right)
    + \frac{\hat{B}}{\Delta t} \left( u_n - u_{n-1} \right) 
    + \hat{A} u_n = \alpha^{-1} f_n ,
  \end{equation}
  where
  \begin{equation} \label{eq:redefineABf_bd}
    \hat{A}  =  \alpha^{-1} A
    - \frac{\beta (1-\cos(\omega \Delta t))}{\alpha \Delta t} B ,
    \qquad \text{ and } \qquad
    \hat{B}   = \alpha^{-1} \beta B .
  \end{equation}
\end{proposition}

\begin{proof}
  To prove this claim,
  $\hat{A}$, $\hat{B}$ and $\hat{c}$
  will be constructed such that
  \begin{equation} \label{eq:recursion_hat}
  \frac{1}{\Delta t^2} \left( u_{n+1} - 2 u_n + u_{n-1} \right)
  + \frac{1}{\Delta t} \hat{B} \left( u_n - u_{n-1} \right) 
  + \hat{A} u_n = \hat{c} f_n ,
  \end{equation}
  if and only if (\ref{eq:time-harmonic_from_semi-discrete}). 
  Inserting $u_n = U e^{i n \omega \Delta t}$ into
  (\ref{eq:recursion_hat}), results in 
  \begin{equation}
    \left[ \frac{2 \cos(\omega \Delta t) - 2}{\Delta t^2}
      + \frac{i \, \sin( \Delta t \omega)}{\Delta t} \hat{B}
      + \frac{1 - \cos( \Delta t \omega)}{\Delta t} \hat{B}
      + \hat{A}
    \right] U e^{i n \omega \Delta t} = \hat{c} e^{i n \omega \Delta t} F .
  \end{equation}
  Using the definitions of $\alpha$ and $\beta$ and multiplying by
  $\alpha$ results in the equivalent equation
  \begin{equation}
    \left[ - \omega^2
      + i \omega \frac{\alpha}{\beta} \hat{B} 
      + \frac{\alpha(1 - \cos( \Delta t \omega))}{\Delta t} \hat{B}
      + \alpha \hat{A} \right] U = \alpha \hat{c} F .
  \end{equation}
  This is equivalent to (\ref{eq:time-harmonic_from_semi-discrete})
  if $\hat{c} = \alpha^{-1}$ and $\hat{A}$ and $\hat{B}$ are defined
  as in (\ref{eq:redefineABf_bd}).
\end{proof}

\begin{definition} \label{def:Iabd}
  Frequency adapted backward differences damped leapfrog will be
  defined as the time integrator given by
  \begin{equation} \label{eq:abd_integrator}
    u_{n+1}
    = I_{\rm abd}(u_n,u_{n-1},f_n) 
    := (2 - K - L) u_n - (I - L) u_{n-1} + g_n  .
  \end{equation}
  where $K$, $L$ and $g_n$ are given by
  \begin{equation} \label{eq:K_L_frequency_adapted_backward_differences}
    \begin{aligned}
      K  = {}& \frac{\Delta t^2}{\alpha} A
      - \frac{\Delta t \,
        \beta (1-\cos(\omega \Delta t))}{\alpha} B \qquad
      \\
      L = {}& \frac{\beta \Delta t}{\alpha} B ,
      \\
      g_n = {}& \frac{\Delta t^2}{\alpha} f_n
    \end{aligned}
  \end{equation}
\end{definition}

We proceed by discussing the choice of $A$, $B$, $\omega$ and $\Delta
t$. Obviously we still have (\ref{eq:A_B_from_H}). The requirements for
$\omega$ and $\Delta t$ should follow from stability conditions for
$I_{\rm abd}$, cf.\ subsection~\ref{subsec:choose_semidiscrete}. These
stability conditions are 
(\ref{eq:stability_cond_K_L}) and 
\begin{equation} \label{eq:stability_cond_bd_upper}
  \text{$4I - K - 2L$ is positive definite} ,
\end{equation}
as shown in subsection~\ref{subsec:analysis_bd} below. 

The method for choosing $\omega$ and $\Delta t$ is somewhat more
complicated than in subsection~\ref{subsec:choose_semidiscrete},
because the stability requirement leads to 
two conditions that both involve $\omega$ and
$\Delta t$.
It is convenient to use $\omega \Delta t$ and $\Delta t$
as parameters instead of $\omega$ and $\Delta t$. The parameter
$\omega \Delta t$ should be between 0 and $\pi$. Given
$\omega \Delta t$ the following expression for $\omega^2$ can be
derived from the condition that $K$ is positive semidefinite and
equation (\ref{eq:K_L_frequency_adapted_backward_differences})
\begin{equation}
  \omega^2 = - \lambda_{\rm min}(\Re H)
  + \frac{\beta}{\omega \Delta t} ( 1 - \cos(\omega \Delta t))
  \lambda_{\rm max}(\Im H) 
\end{equation}
(instead of $\lambda_{\rm min}(\Re H)$ and
$\lambda_{\rm max}(\Im H) $ lower and upper bounds can be used respectively).
From the condition that $4I - K - 2L$ is positive definite we then get
the following scalar condition
\begin{equation} \label{eq:Deltat_bd_full_inequality}
  \frac{ (\omega \Delta t)^2} {\alpha}
  \left( \frac{\lambda_{\rm max}(\Re H)}{\omega^2} + 1 \right)
  + \frac{\beta}{\alpha} (1 + \cos(\omega \Delta t))
  \frac{\lambda_{\rm max}(\Im H)}{ \omega^2} < 4 .
\end{equation}
The following is a stronger inequality than
(\ref{eq:Deltat_bd_full_inequality})
\begin{equation}
\left( \frac{\lambda_{\rm max}(\Re H)}{-\lambda_{\rm min}(\Re H)} + 1
\right) (\omega \Delta t)^2
+ 2 \frac{\lambda_{\rm max}(\Im H)}{-\lambda_{\rm min}(\Re H)}
\omega \Delta t < 4 .
\end{equation}
From here a value of $\omega \Delta t$ can be obtained that satisfies
the conditions by solving a simple quadratic equation.
If a larger value of $\omega \Delta t$ is desired, one can look numerically
for a value as large as possible for which
(\ref{eq:Deltat_bd_full_inequality})
is still satisfied.

The time-domain approximate solution operator is defined similarly as
in Definition~\ref{def:approx_solver}.
\begin{definition} \label{def:approx_solver_bd}
  Let $\chi$ be an admissible $C^\infty$ window function 
  and let $T$ be a positive real constant, such that
  $n_{\rm steps} := 2\pi \omega^{-1} T / \Delta t$ is an integer.
  For $F \in \CC^N$, let
  \begin{equation} 
    f_n = f(n \Delta t) , \qquad
    f(t) = \chi(1 - \frac{t}{2\pi \omega^{-1} T} ) e^{i \omega t} F .
  \end{equation}
  The {\em time-domain approximate solution operator for $H$} associated with the
  integrator $I_{\rm abd}$ is the linear map $S_T^{\rm abd} : \CC^N \to \CC^N$ defined by
  \begin{equation}
    S_T^{\rm abd} F = e^{- i 2\pi T} u_{n_{\rm steps} } , 
  \end{equation}
  where $u_n$, $n=0,1, \ldots, n_{\rm steps} $ is given by
  \begin{equation}
    u_{n+1} = I_{\rm abd} (u_n,u_{n-1}, f_n) , \qquad u_0 = 0 .
  \end{equation}
\end{definition}

\subsection{Analysis}
\label{subsec:analysis_bd}

To establish stability of backward differences damped leapfrog, one can study
the growth of solutions to the recursion
\begin{equation} \label{eq:homogeneous_recursion_bd}
  u_{n+1} + (- 2+K + L) u_n + (I - L) u_{n-1} = 0 
\end{equation}
using the energy function
\begin{equation}
  \begin{aligned}
    E_{\rm bd}(n-1/2) = {}& \langle u_n - u_{n-1} , (4I - K)  (u_n - u_{n-1}  ) \rangle
    + \langle u_n + u_{n-1} , K (u_n + u_{n-1}  ) \rangle
    \\
    {}& - 2 \langle u_n - u_{n-1}, L (u_n - u_{n-1}) \rangle .
  \end{aligned}
\end{equation}
The following conclusions can be drawn.
  \begin{enumerate}[(i)]
  \item
    If $K$ and $4I - K - 2L$ are positive definite, then $E_{\rm bd}$ is equivalent
    to a norm on $\RR^{2N}$. If $L = 0$ then $E_{\rm bd}$ is conserved and
    solutions remain bounded if $t \to \pm \infty$. If $L$ is positive
    semidefinite then $\Delta E_{\rm bd}(n) \le 0$
    and solutions remain bounded if $t \to \infty$.
  \item
    If instead $K$ is positive semidefinite with one or more zero eigenvalues
    and $4I - K - 2L$ is positive definite, then $E_{\rm bd}$ is not equivalent to a
    norm. If $L = 0$ then $E_{\rm bd}$ is
    conserved and solutions grow at most linearly if $t \to \pm \infty$,
    If $L$ is positive semidefinite, then $\Delta E_{\rm bd}(n) \le 0$ and
    solutions grow at most linearly if $t \to \infty$.
  \end{enumerate}

For $S_{\rm abd}$ results similar to Theorems \ref{th:convolution_formula_S_T_acd} and
\ref{th:convergence_acd} can be
obtained by following the same method of proof.
The equivalent of (\ref{eq:operator_Cacd}) is 
\begin{equation} \label{eq:operator_Cabd}
  C_{\rm abd}(u)(t) = \frac{\alpha}{\Delta t^2}
  \left[ u(t+\Delta t) + (- 2+K+L) u(t)
    + (I - L) u(t-\Delta t) \right] 
\end{equation}
where $K,L$ are as defined in (\ref{eq:K_L_frequency_adapted_backward_differences}).
A causal Green's function $\Phi_{\rm abd}$ is defined satisfying
\begin{equation} 
  C_{\rm abd} \Phi_{\rm abd} (t) = I \delta(t) ,
  \qquad \text{and} \qquad 
  \Phi_{\rm abd}(t) = 0 \text{ if } t \le  0 ,
\end{equation}
The following theorems are proved in the same ways
as theorems~\ref{th:convolution_formula_S_T_acd} and
\ref{th:convergence_acd}

\begin{theorem} \label{th:convolution_formula_S_T_abd}
Assume $\psi_\epsilon$ and $\Phi_{\rm abd}$ are as just defined, then 
\begin{equation} \label{eq:convolved_Fourier_transform_abd}
  S_{T}^{\rm abd} =  \psi_{1/T} \ast \widehat{\Phi}_{\rm abd}(\omega) .
\end{equation}
\end{theorem}

\begin{theorem} \label{th:convergence_abd}
If
  \begin{equation} \label{eq:H_non-singular_abd}
    \text{$H = -\omega^2 I + i \omega B + A$ is non-singular}
  \end{equation}
  and $A,B,\omega$ and $\Delta t$ are such that $K$, $L$ defined in
  (\ref{eq:K_L_frequency_adapted_backward_differences}) satisfy
  the stability conditions (\ref{eq:stability_cond_K_L}) and (\ref{eq:stability_cond_bd_upper})
  then
  \begin{equation}
    \lim_{T \to \infty} S_T^{\rm abd} = H^{-1} .
  \end{equation}
\end{theorem}

\section{Concluding remarks}

In this paper we constructed a time-domain preconditioner for 
indefinite linear systems that come from discretizing time-harmonic
wave equations, and studied its properties and behavior
analytically and with numerical examples.  It should be emphasized
that the method does not compute in the physical time domain. We call
it a time-domain method, because of the similarity with classical
time-domain methods for time-harmonic waves.
For the practical application it would be useful to have further
examples, with different matrices $H$ and a comparison with
alternatives.
We will make a few brief remarks in this direction.

First it is clear that the method requires relatively little
memory, compared to
alternatives that use LU (or LDL${}^T$) decompositions, such as 
domain-decomposition and direct methods, see
\cite{stolk2017improved,poulson2013parallel} and  references in
\cite{gander2019class}.
In the 3-D implementation here, with GMRES with restart $m = 10$,
most memory was used for the approximately
$m + 2$ complex double
precision vectors needed for GMRES. 
This was in part because the timestepping was done in single precision, and
using real fields. In principle, memory use could be somewhat reduced by setting
$m = 5$ or using a different iterative method like BiCGSTAB. 

For finite element discretizations the cost in general will be
different. These methods often have stricter CFL bounds, there may
be more nonzero matrix elements, and the use of unstructured meshes
may also influence efficiency. On the other hand, the size
of the grid cells can be adapted to the local velocity so that the
number of grid cells can be smaller. 

It is difficult to compare performance of different methods, as
methods are run on different computer systems, with different examples
and implementations are optimized to different degrees. We refer to
\cite{appelo2019waveholtz,CalandraEtAl2013,liu2018solving,
  poulson2013parallel,stolk2017improved,taus2019sweeps,WangDeHoopXia2011}
for some alternative methods. We believe that computation times of the
method as outlined are modest, and there is potential for further
improvements, by using GPUs or by reorganizing the code to better
make use of cached data.  It remains challenging to get the most out
of modern computer hardware in time-domain finite-difference and
finite-element simulations. We hope that recent developments in this
area, cf.\ \cite{louboutin2019devito}, will lead to further improvements.

\bibliographystyle{abbrv}
\bibliography{helmfdtd}

\appendix

\section{Additional material for subsection~\ref{subsec:optimized-fd}}
\label{app:opt_fd}

In case of variable coefficients, the 27 point optimized finite-differences discretization used in
subsection~\ref{subsec:optimized-fd} is done in a quasi-finite-element
way that we now explain. This is consistent with
\cite{stolk2016dispersion}.

We first define some notation. In this
appendix, three dimensional indices are denoted by greek letters,
e.g.\ $\alpha = (\alpha_1,\alpha_2,\alpha_3)$. A grid cell will have
the same index as the point in the lower (in all dimensions) corner.
The set of corners of a grid cell will be denoted by $C(\alpha)$. We
define
\begin{equation}
  \Delta(\alpha,\beta)
  = | \alpha_1 - \beta_1 |
  + | \alpha_2 - \beta_2 |
  + | \alpha_3 - \beta_3 | .
\end{equation}
If $\alpha, \beta$ are two corner points of a grid cell, then this
number is 0 if they are the same and 1, 2, or 3 if they are opposite
points on an edge, face, or the cell itself respectively.
We define
\begin{equation}
  \tilde{f}_s\left( \frac{k h}{2\pi} \right) 
  = \frac{1}{2^{3-s}} f_s\left( \frac{k h}{2\pi} \right) 
\end{equation}
In case of variable coefficients, the coefficient $k$ will be constant
on grid cells, its value is denoted by $k^{(\alpha)}$ 

The matrix $\Re H$ will be associated with a bilinear form
\begin{equation}
  \mathcal{H}(v,u) =
  \sum_{\text{cells $\alpha$}} \mathcal{H}_{\rm cell}(\alpha; v,u) .
\end{equation}
The contribution for a single cell is given by
\begin{equation}
  \mathcal{H}_{\rm cell}(\alpha; v,u)
  =
  \sum_{\beta,\gamma \in C(\alpha)}
  \tilde{f}_{\Delta(\beta,\gamma)} \left( \frac{h k^{(\alpha)}}{2 \pi}
  \right)
  v^{(\beta)} u^{(\gamma)}
\end{equation}
For each cell we have
\begin{equation}
  \mathcal{H}_{\rm cell}(\alpha; u,u)
  \ge  - (k^{(\alpha)})^2 \frac{1}{8} \sum_{\beta \in C(\alpha)}
  \left| u^{(\beta)} \right|^2 .
\end{equation}
An expression for $\lambda_{\rm min}(\Re H)$ follows straightforwardly
from this.  For an upperbound for
$\lambda_{\rm max}(\Re H)$ one can use (\ref{eq:upperbound_HoptimFd})
with $k$ replaced by $k_{\rm min}$.

\section{Fourier transform of the unit step function on $\ZZ$}
\label{app:Fourier_transform_unit_step}

In this appendix we consider the Fourier transform of a discrete unit step
function. Although this is a standard result, we could not locate a
proof and included one here.
\begin{proposition} \label{prop:FT_unit_step}
Denote by $U(n)$, $n \in \ZZ$ the unit step function
\begin{equation} \label{eq:define_unit_step}
  U(n) =
  \left\{\begin{array}{ll}
           0 & \text{if $n\le 0$}\\
           1 & \text{if $n>0$}
           \end{array}\right.
\end{equation}
The Fourier transform of $U$ equals
\begin{equation} \label{eq:FT_unit_step}
  \widehat{U}(\nu) = \pi \delta(\nu) - \frac{i}{2} \cot(\frac{\nu}{2}) - \frac{1}{2} .
\end{equation}
Here the distribution associated with $\cot$ is defined by the principal value integral.
\end{proposition}

\begin{proof}
Consider the first the Fourier transform of $g(n) = \sgn(n)$. We claim it is
given by
\begin{equation} \label{eq:FT_of_sgn_1}
  \widehat{g}(\nu) = \frac{1}{1-e^{-i \nu}} - \frac{1}{1-e^{i \nu}}  ,
\end{equation}
which is interpreted as a distribution using the principal value integral.
Indeed, consider the inverse Fourier transform of our candidate for $\widehat{g}$
\begin{equation}
  \frac{1}{2 \pi} \int_{-\pi}^\pi \left(
    \frac{e^{in \nu}}{1-e^{-i \nu}} - \frac{e^{i n \nu}}{1-e^{i \nu}}
      \right) \, d \nu .
\end{equation}
For each of the two contributions the real part is even and the
imaginary part is odd in $\nu$. Therefore, for each of the two
contributions only the real part contributes to the integral and the
integral equals
\begin{equation}
\int_{-\pi}^\pi \frac{e^{i n \nu} - e^{-i n \nu}}{1-e^{-i \nu}} \, d
  \nu
  = \left\{
    \begin{array}{ll}
      0 & \text{if $n=0$} \\
      \frac{1}{2\pi} \int_{-\pi}^\pi \left( e^{in\nu} + e^{i(n-1)\nu}
      + \ldots + e^{-i(n-1) \nu} \right) \, d \nu = 1 & \text{if $n>0$}
      \\
      - \frac{1}{2\pi} \int_{-\pi}^\pi \left( e^{i(-n)\nu} + e^{i(-n-1)\nu}
      + \ldots + e^{-i(-n-1) \nu} \right) \, d \nu = -1 & \text{if $n<0$}
    \end{array}
  \right.
\end{equation}
Simple manipulations show that $\widehat{g}$ can also be written as
$\widehat{g} = - i \cot( \frac{\nu}{2} )$.
The unit step function can be written as
\begin{equation}
  U(n) = \frac{1}{2} + \frac{1}{2}\sgn(n) - \frac{1}{2} \delta_{0,n}
\end{equation}
The Fourier transform is hence as given in (\ref{eq:FT_unit_step}).
\end{proof}

\end{document}